\documentclass{scrartcl}
\usepackage[utf8]{inputenc}
\usepackage[T1]{fontenc}
\usepackage{lmodern}
\usepackage[english]{babel}
\usepackage{mathtools,amssymb,amsfonts,amsthm}
\usepackage{enumitem}

\usepackage{hyperref}
\usepackage{xcolor}
\definecolor{green}{RGB}{0,50,0}

\newtheorem{theorem}{Theorem}[section]
\newtheorem{proposition}[theorem]{Proposition}
\newtheorem{lemma}[theorem]{Lemma}
\newtheorem{corollary}[theorem]{Corollary}
\theoremstyle{definition}
\newtheorem{definition}[theorem]{Definition}

\newtheorem{hypothesis}[theorem]{Hypothesis}

\theoremstyle{remark}
\newtheorem{remark}[theorem]{Remark}
\usepackage{lineno}
\DeclareMathOperator{\R}{\mathbb{R}}
\DeclareMathOperator{\Z}{\mathbb{Z}}

\DeclareMathOperator{\X}{\mathbb{X}}
\DeclareMathOperator{\Y}{\mathbb{Y}}

\newcommand{\Banach}{\R^m}
\newcommand{\Banachdual}{{\R^m}}
\newcommand{\domE}{\mathbb{D}}
\newcommand{\reynold}{\mathfrak{R}}

\newcommand{\E}{\mathcal{E}}

\newcommand{\T}{{\mathbb{T}^d}}

\DeclareMathOperator{\dom}{dom}

\DeclareMathOperator{\C}{\mathcal{C}}

\DeclareMathOperator{\N}{\mathbb{N}}
\DeclareMathOperator*{\esssup}{ess\,sup}
\newcommand{\sym}{\text{sym}}
\newcommand{\skw}{\text{skw}}

\DeclareMathOperator*{\argmin}{arg\,min}
\newcommand{\dreidots}{\text{\,\multiput(0,-2)(0,2){3}{$\cdot$}}\,\,\,\,}

\DeclareMathOperator{\ra}{\rightarrow}
\newcommand{\de}{\,\mathrm{d}}
\DeclareMathOperator{\tr}{tr}

\newcommand{\f}[1]{{\pmb{ #1}}}
\DeclareMathOperator{\di}{div}

\newcommand{\curl}{\nabla \times}

\DeclareMathOperator{\interi}{int}

\newcommand{\setdom}{M}

\newcommand{\tU}{\tilde{\f U}}
\renewcommand{\th}{\tilde{h}}
\newcommand{\tm}{\tilde{\f m}}

\renewcommand{\t}{\partial_t}

\DeclareMathOperator{\BV}{{BV}([0,T])}
\newcommand{\ov}[1]{\overline{{#1}}}
\newcommand{\un}[1]{\underline{{#1}}}
\renewcommand{\O}{{\T}}
\newcommand{\F}[2]{{\int_{\T} \f F({ #1}): \nabla  {#2} \de \f x }}

\newcommand{\transpose}{T}
\DeclareMathOperator{\dv}{div}
\newcommand{\dens}{h}
\newcommand{\mom}{\f m}
\newcommand{\denst}{\rho}
\newcommand{\momt}{\f \varphi}
\newcommand{\pres}{p}
\newcommand{\pot}{P}
\newcommand{\dx}{\mathrm{d}\f x}
\newcommand{\dz}{\mathrm{d}\f z}
\newcommand{\ddz}{\frac{\mathrm{d}}{\mathrm{d}z}}
\newcommand{\ddt}{\frac{\mathrm{d}}{\mathrm{d}t}}
\newcommand{\dt}{\mathrm{d}t}
\newcommand{\np}[1]{(#1)}
\newcommand{\nb}[1]{[#1]}
\newcommand{\bp}[1]{\big(#1\big)}
\newcommand{\bb}[1]{\big[#1\big]}
\newcommand{\Bp}[1]{\bigg(#1\bigg)}
\newcommand{\Bb}[1]{\bigg[#1\bigg]}
\newcommand{\set}[1]{\{#1\}}
\newcommand{\setl}[1]{\big\{#1\big\}}

\newcommand{\setc}[2]{\{#1\mid #2\}}
\newcommand{\setcl}[2]{\big\{#1\bigm\vert #2\big\}}
\newcommand{\setcL}[2]{\Big\{#1\Bigm\vert #2\Big\}}
\newcommand{\snorm}[1]{\lvert #1 \rvert}

\newcommand{\snormL}[1]{\Big\lvert #1 \Big\rvert}

\newcommand{\norml}[1]{\big\lVert #1 \big\rVert}

\newcommand{\LR}[1]{L^{#1}}

\newcommand{\LRsigma}[1]{L^{#1}_\sigma}

\numberwithin{equation}{section} 
\title{Existence of energy-variational solutions to hyperbolic conservation laws}

\author{Thomas Eiter\footnotemark[1] \and
Robert Lasarzik%
\footnote{Weierstrass Institute for Applied Analysis and Stochastics,
Mohrenstr. 39, 10117 Berlin, Germany,
\newline
\texttt{thomas.eiter@wias-berlin.de}
\newline
\texttt{robert.lasarzik@wias-berlin.de}
}}
\date{\today}
\begin{document}

\maketitle
\begin{abstract}
We introduce the concept of energy-variational solutions for hyperbolic conservation laws.
Intrinsically, 
these energy-variational solutions fulfill the weak-strong uniqueness
principle and the semi-flow property,
and the set of solutions is convex and weakly-star closed. 
The existence of energy-variational solutions is proven 
via a suitable time-discretization scheme
under certain assumptions.
This general result 
yields 
existence of energy-variational solutions 
to the magnetohydrodynamical equations for ideal incompressible fluids
and to the Euler equations in both the incompressible 
and the compressible case.
Moreover, we show that energy-variational solutions to the Euler equations
coincide with dissipative weak solutions.
\end{abstract}


\tableofcontents
\section{Introduction}
Hyperbolic conservation laws form a class of nonlinear evolution equations that is omnipresent in mathematical physics and its applications. 
These range from traffic models~\cite{traffic} over thermomechanics~\cite[Sec.~2.3]{dafermos2} to fluid dynamics and weather forecast~\cite{applfluid}. Even though this class of equations is so fundamental and plays such a prominent role in the research of partial differential equations, up to now there is no 
suitable concept
of generalized solutions
such that existence can be established for a large class of general multi-dimensional 
hyperbolic conservation laws.
 To contribute to filling this gap, in this article we 
 propose the concept of energy-variational solutions.

  We consider general conservation laws
  \begin{subequations}\label{eq}
   \begin{alignat}{2}
   \label{eq.pde}
 \t \f U + \di \f F(\f U) ={}& \f 0  && \quad\text{in }\mathbb{T}^d\times (0,T)\,,\\
 \label{eq.iv}
 \f U(\cdot,0) ={}& \f U_0 && \quad\text{in }\mathbb{T}^d\,
 \end{alignat}
   \end{subequations}
on the $d$-dimensional (flat) torus $\T$, $d\in\N$,
and for a finite time $T\in(0,\infty)$.
Here $\f U\colon\T\times(0,T)\to\R^m$, $m\in\N$, denotes the 
unknown state variable, 
$\f F : \R^m \ra \R^{m\times d }$ is a given flux matrix
depending on the state,
and $\f U_0\in\R^m$ denotes prescribed initial data.
As usual (\textit{cf.}~\cite[Sec.~11.4.2]{evans}), we assume that there exists a strictly convex entropy  $ \eta : \R^m \ra [0,\infty ] $
such that the total entropy 
 $ \E (\f U(t )) := \int_{\T} \eta( \f U(t)) \de \f x $
is conserved along smooth solutions,
but which may decrease along non-smooth solutions.
To ensure this, we assume that 
$$ \F{ \tU }{D \eta (\tU)} = 0  $$
  for all suitable $\tU $.
This condition differs from the usual entropy-pair assumption,
 where the existence of a corresponding entropy flux is required,
 but it allows for more general entropy functions
 and therefore a larger class of conservation laws; 
see Remark \ref{rem:integralcondition} below for further explanation.
Observe that we use the letter $\E$ to denote the
total entropy since in the considered examples
the mathematical entropy is always played by the physical energy of the 
respective system. 

Hyperbolic conservation laws are well understood in one spatial dimension,
that is, in the case $d=1$ or $m=1$. 
Going back to the fundamental works of Hopf~\cite{hopf} and Lax~\cite{lax}, the theory is nowadays fairly standard; see~\cite{evans} and~\cite{dafermos2} for example. 
In contrast,
the one-dimensional theory cannot be transferred to the multi-dimensional case
$m,d\geq 2$ immediately,
where a general solution concept that ensures solvability is missing. 
Instead, solution concepts are usually constructed
such that they fit to one specific conservation law,
and often there are several different concepts for the same equation.

A prominent example is the Euler system for inviscid fluid flow,
for which 
DiPerna and Majda established the existence of measure-valued solutions 
in the incompressible case~\cite{DiPernaMajda},
and a weak-strong uniqueness principle
was proven later in~\cite{weakstrongeuler}.
Weak-strong uniqueness is another favorable property for any solution concept
and means that a generalized solution
coincides with a strong solution with the same initial data if the latter exists.
In the same article~\cite{weakstrongeuler}, the weak-strong uniqueness of measure-valued solutions to hyperbolic conservation laws was shown, 
but the existence of these solutions is not known and not expected to hold in general. 
The weak-strong uniqueness principle for dissipative measure-valued solutions,
where the measure-valued formulation is enriched with a defect measure,
was shown for more general conservation laws in~\cite{Gwiazda}, 
but still their existence remains unclear. 
In the case of the compressible Euler equations, 
the existence of dissipative weak solutions, 
defined by enriching the weak formulation with a defect measure,
was shown 
in~\cite{BreitComp}, and a weak-strong uniqueness principle was proved in~\cite{weakstrongCompEul}.

We shall see that both the incompressible and the compressible Euler equations
can be treated in the abstract framework of hyperbolic conservation laws presented here.
In particular, we establish existence of energy-variational solutions to both systems,
and we show that they coincide with the corresponding dissipative weak solutions.
In this respect, we present a new way 
to construct 
dissipative weak solutions for these equations.
As another example, we consider the equations of magnetohydrodynamics 
for an incompressible ideal fluid, which means that the effects of viscosity
and electrical resistivity are neglected.
While there are results on 
the local existence of strong solutions~\cite{Schmidt1988,Secchi1993,DiazLerena2002},
and a weak-strong uniqueness principle for measure-valued solutions was shown in~\cite{Gwiazda},
the global existence of suitably generalized solutions seems to be unknown.
By providing existence of energy-variational solutions to this system,
the present work gives the first result in this direction.
We believe that the class of equations considered here is quite general,
and that the presented theory yields existence results for many other conservation laws.

To explain the main idea of our solution concept, 
let us begin with 
the classical approach towards a generalized solution concept for problem \eqref{eq},
namely the notion of weak solutions,
defined via the weak formulation of \eqref{eq.pde},
that is, the identity 
\begin{equation}\label{eq:weak.intro}
- \langle   \f U, \Phi  \rangle \Big|_{s}^t  + \int_s^t \int_{\T}  \f U \cdot\t \Phi + \f F (\f U ) : \nabla \Phi\de \f x  \de \tau = 0 \,
\end{equation}
for $s,t\in[0,T]$ and all test functions $\Phi$
in a suitable class $\Y$ of test functions.
As mentioned above, 
a natural assumption is that the total entropy is non-increasing along solutions,
which means that $\E(\f U)\big|_s^t\leq 0$ if $s<t$.
Combing this condition with \eqref{eq:weak.intro},
we obtain the variational inequality
\begin{equation}\label{eq:weak.var}
\bb{\E(\f U)- \langle  \f U ,\Phi \rangle }\Big|_{s}^t  + \int_s^t \int_{\T} \f U \cdot  \t \Phi  + \f F (\f U ) : \nabla \Phi\de \f x  \de \tau \leq 0\,
\end{equation}
for $s<t$ and $\Phi\in\Y$. 
Since \eqref{eq:weak.intro} can be recovered from
\eqref{eq:weak.var}
(see also Lemma \ref{lem:var.affine} below),
we may also take \eqref{eq:weak.var} to define 
weak solutions with non-increasing total entropy.
As explained above,
existence of such weak solutions cannot be guaranteed for general 
hyperbolic conservation laws,
which is why we introduce the concept of 
energy-variational solutions.
The main idea is 
to replace the total mechanical entropy $\E(\f U) \in L^\infty(0,T)$ with an auxiliary entropy variable $E\in\BV$,
which may be seen as a turbulent entropy and
may exceed the mechanical entropy of the system.
Additionally, we introduce the difference $\E(\f U)-E\leq 0$,
weighted by a suitable factor $\mathcal K(\Phi)\geq 0$ 
depending on the test function,
into the equation \eqref{eq:weak.var}.
This leads to the inequality 
\begin{equation}
\left[  E - \langle \f U , \Phi  \rangle\right ] \Big|_{s}^t  + \int_s^t \Bb{\int_{\T}\f   U \cdot \t \Phi  + \f F (\f U ) : \nabla \Phi\de \f x + \mathcal{K}(\Phi) \left [\E (\f U) - E \right ]}  \de \tau \leq 0 \, \label{eq:envar.intro}
\end{equation}
for $s<t$ and $\Phi\in\Y$,
which will serve as the basic inequality defining
energy-variational solutions.
In particular, 
if we have $E=\E(\f U)$, then \eqref{eq:weak.var}
is equivalent to \eqref{eq:envar.intro},
and energy-variational solutions 
coincide with weak solutions.
The crucial assumption for our approach
is that the function $\mathcal K$ is chosen in such 
a way that the mapping
\[
\f U \mapsto \F{\f U}{\Phi} + \mathcal{K}(\Phi) \E(\f U) 
\]
is convex for any $\Phi\in\Y$.
Under this assumption, $(\f U,E)$ appears in \eqref{eq:envar.intro}
in a convex way,
so that inequality~\eqref{eq:envar.intro} is preserved
under weak$^*$ convergence. 

Note that the idea of relaxing the formulation of an evolution equation to a variational inequality
and providing convexity by introducing an additional term
goes back to Pierre-Louis Lions
in the context of the incompressible Euler equations~\cite[Sec.~4.4]{lionsfluid}. 
Similar solution concepts have recently been used in the context of 
fluids with viscosity
as the incompressible Navier--Stokes equations~\cite{maxidss}
and viscoelastic fluid models~\cite{EiHoLa22}.

Besides showing existence 
of energy-variational solution via a semi-discretization in time,
which may justify their usefulness for 
numerical implementations,
we further show certain properties 
that are directly included in the solution concept,
for example, a weak-strong uniqueness principle.
Furthermore, we introduce the concept of energy-variational solutions 
in such a way that the semi-flow property is satisfied. 
This is a desirable property of a solvability concept, 
in particular, when uniqueness of solutions cannot be guaranteed; 
see~\cite{basaric,BreitComp} for example. 

As is the case for many generalized solution concepts, 
energy-variational solutions may not be unique but 
instead 
capture all limits of suitable approximations.
Hence, additional selection criteria would have to be applied in order to choose the physically relevant solution. This definitely requires further research, but we shall see that the class of energy-variational solutions has desirable properties for such a selection process.
In particular, we prove that the set of energy-variational solutions is convex and weakly$^\ast$ closed, which might make it possible to define an appropriate minimization problem on this set (\textit{cf.}~\cite{envar}),
and to identify the (unique) minimizer 
with the physically relevant solution.
For scalar conservation laws, Dafermos~\cite{dafermosscalar} proposed the entropy-rate admissibility criterion to select the physically relevant solution. He was able to prove that in a certain class this selection procedure coincides with a selection according to the well established Lax-admissibility criterion~\cite{dafermosscalar}. 
It is worth noticing that for the auxiliary variable $E\in \BV$ the entropy rate $\t E$ is well defined in the space of Radon measures, and the proposed minimization of this value may be defined at least for finitely many points in time.
 Therefore, it might be possible to follow Dafermos's proposed criterion in the present case. 
This is in accordance with the semi-discrete time-stepping scheme proposed in~\eqref{eq:timedis} below, where the energy is minimized in every step,
which might provide additional regularity for the minimizer as well as for the solution in the limit. 
This question will be further investigated in the future, together with the performance of the proposed semi-discretization in numerical experiments.   

The article is organized as follows: In Section~\ref{sec:pre}, we explain the relevant notation and
introduce the notion of energy-variational solutions for hyperbolic conservation laws.
We formulate the main result on their existence 
and collect several auxiliary lemmas. 
Section~\ref{sec:hyper} is concerned with the study of energy-variational solutions to these hyperbolic conservation laws. 
We derive a number of general properties of energy-variational solutions,
and we prove the existence of energy-variational solutions via the convergence of a suitable time-discretization based on an iterative minimization procedure. 
After considering the incompressible hydrodynamical equations and the incompressible Euler equations in Section~\ref{sec:incomp}, we deal with the compressible Euler equations in Section~\ref{sec:comp}.

\section{Preliminaries and main result\label{sec:pre}}
\subsection{Notation}
For $d\in\N$, 
we denote the scalar product of 
two vectors $\f a, \f b\in\R^d$
by $\f a \cdot \f b\coloneqq\f a_j \f b_j$,
and the Frobenius product of two matrices 
$\f A,\f B\in \R^{m\times d}$
by 
$\f A : \f B \coloneqq \f A_{ij}\f B_{ij}$.
Here and in the following, we tacitly use Einstein summation convention
and implicitly sum over repeated indices from $1$ to $d$ or $m$ depending on the context.
By
$\mathbb{R}^{d\times d}_{\sym}$, 
$\R^{d\times d}_{\skw}$ and
$\mathbb{R}^{d\times d}_{\sym,+}$
we denote the sets of symmetric, skew-symmetric and
symmetric positive semi-definite $d$-dimensional matrices, respectively.
The symbols $(\f A)_{\sym}=\frac{1}{2}(\f A+\f A^T)$ 
and $(\f A)_{\skw}=\frac{1}{2}(\f A-\f A^T)$ denote the symmetric and 
the skew-symmetric part
of a matrix $\f A\in\R^{d\times d}$,
and by $(\f A)_{\sym,+}$ and $(\f A)_{\sym,-}$, we denote the 
positive semi-definite and the negative semi-definite part of the symmetric 
matrix $(\f A)_{\sym}$, respectively.
We usually equip matrix spaces with the spectral norm $\snorm{\cdot}_2$
defined by
\begin{equation}\label{eq:spectralnorm}
\snorm{\f A}_2= \sup _{|\f a| = 1} \f a ^T \cdot \f A \f a\,,
\end{equation}
that is, $\snorm{\f A}_2$ is the square root of the largest eigenvalue 
of $\f A^\transpose \f A$.
The dual norm of the spectral norm with respect to the Frobenius product 
is the trace norm and denoted by $| \cdot |'_2$.
For symmetric matrices 
$\f S\in  \mathbb{R}^{d\times d}_{\sym}$
we thus have 
$
\snorm{\f S}_2
= \max_{j\in\{ 1,\ldots,d\}} \snorm{\lambda_j}
$
and 
$| \f S |'_2 = \sum_{i=j}^d \snorm{\lambda _j} $,
where $\lambda _j$, $j=1,\dots,d$, are the (real) eigenvalues of the matrix $\f S$. 
For symmetric positive semi-definite matrices $\f S \in \mathbb{R}^{d\times d}_{\sym,+}$ we may write
$| \f S |'_2 = \sum_{i=j}^d {\lambda _j} = \f S:I = \tr(\f S)$,
where $I$ denotes the identity matrix in $\R^{d\times d}$.

By $\T\coloneqq\R^d/\Z^d$ we denote the $d$-dimensional (flat) torus
equipped with the Lebesgue measure. 
The Radon measures on $\T$ taking values in $\mathbb{R}^{d\times d}_{\sym}$ are denoted by $\mathcal{M}(\T ; \mathbb{R}^{d\times d}_{\sym} ) $, which may be interpreted as the dual space of the corresponding continuous functions, \textit{i.e.,} $\mathcal{M}(\T; \mathbb{R}^{d\times d}_{\sym} ) =(\C(\T; \mathbb{R}^{d\times d}_{\sym} ) )^*$.
Moreover, $\mathcal{M}(\T;  \mathbb{R}^{d\times d}_{\sym,+} ) $
is the class of symmetric positive semi-definite Radon measures,
which consists of Radon measures $\mu \in \mathcal{M}(\T;  \mathbb{R}^{d\times d}_{\sym} ) $
such that for any $\f \xi \in\R^d$ the measure $ \f \xi \otimes \f \xi : \mu $ is nonnegative. 

For a Banach space $\mathbb{X}$, we denote its dual space by $\mathbb{X}^\ast$,
and we use $\langle\cdot,\cdot\rangle$ to denote the associated dual pairing.
The space $\C_w([0,T];\mathbb X )$ denotes the class of functions on $[0,T]$ taking values in $\mathbb X$ that are continuous with respect to the weak topology of $\mathbb X$.
Analogously,  the space $\C_{w^*}([0,T];\mathbb X^* )$ denotes the class of functions on $[0,T]$ taking values in $\mathbb X^*$ that are continuous with respect to the weak$^*$ topology of $\mathbb X^*$.
The space $L^\infty_{w^*} ([0,T];\mathbb X^*)$ is the space of all function  on $[0,T]$ taking values in $\mathbb X^*$ that are Bochner measurable and essentially bounded with respect to $\mathbb X^*$ equipped with the weak$^\ast$ topology.  

We write $x_n\rightharpoonup x$ 
if a sequence $(x_n)\subset \mathbb X$ 
converges weakly to some $x\in \mathbb X$,
and $\varphi_n\xrightharpoonup{\ast}\varphi$ 
if a sequence $(\varphi_n)\subset \mathbb X^\ast$ 
converges weakly$^\ast$ to some $\varphi\in \mathbb X^\ast$.
In spaces of the form
$L^\infty(0,T;\X)$
we usually consider a mixture of the weak convergence in $\X$
and weak$^\ast$ convergence in $L^{\infty}$,
which we call weak$(^\ast)$ convergence,
and we write $u_n\xrightharpoonup{(\ast)} u$ 
if a sequence $(u_n)\subset L^\infty(0,T;\X)$
converges weakly$(^\ast)$ to some $u\in L^\infty(0,T;\X)$,
that is,
if
\begin{equation}
\label{eq:weakconv.LinfL1}
\forall f\in\LR{1}(0,T;\mathbb X^\ast):
\quad
\lim_{n\to\infty}\int_0^T \langle u_n(t), f(t)\rangle \,\dt
=\int_0^T \langle u(t), f(t)\rangle \,\dt.
\end{equation}

The total variation of a function $E:[0,T]\ra \R$ is given by 
$$ | E |_{\text{TV}([0,T])}= \sup_{0=t_0<\ldots <t_n=T} \sum_{k=1}^n \lvert E(t_{k-1})-E(t_k) \rvert\,, $$
where the supremum is taken over all finite partitions of the interval $[0,T]$. 
We denote the space of all integrable functions on $[0,T]$ 
with bounded variation by~$\BV$, and we equip this space with the norm
$\| E \|_{\BV} :=  \| E\|_{L^1(0,T)} + | E |_{\text{TV}([0,T])}$~(cf.~\cite{BV}). 
Recall that an integrable function $E$ has bounded variation if and only if its
distributional derivative $E'$ is an element of 
$\mathcal M([0,T])$, the space of finite Radon measures on $[0,T]$.
Moreover, $\BV$ coincides with the dual space of a Banach space,
see \cite[Remark~3.12]{AmbrosioFusoPallara_BVFunctions_2000} for example,
and we usually work with the corresponding weak$^\ast$ convergence,
which can be characterized by
\[
E_n \xrightharpoonup{*} E \text{ in } \BV \quad \iff \quad
E_n \to E \text{ in } L^1(0,T) \ \text{ and } \  
E_n'\xrightharpoonup{*} E' \text{ in } \mathcal M([0,T]).
\]
Note that the total variation of a decreasing non-negative function $E$ 
can be estimated by the initial value since
\[
| E|_{\text{TV}([0,T])} = \sup_{0=t_0<\ldots <t_n=T}\sum_{k=1}^N \bp{ E(t_{k-1})-E(t_k) }
= E(0) - E(T) \leq E(0) \,.
\]

Let $\eta :\R^d  \to [0,\infty]$ be a convex, lower semi-continuous function with $\eta (\f 0 )= 0$. 
The domain of $\eta$ is defined by $\dom\eta=\setc{\f x\in\R^d}{\eta(\f x)<\infty}$.
We denote the convex conjugate of $\eta$ by $\eta^\ast$, 
which is defined by 
\begin{align*}
\eta^*(\f z) =\sup_{\f y \in \Banach} \left [
\f z\cdot\f y
- \eta(\f y)\right ] \qquad \text{for all }\f z \in \Banachdual \,.
\end{align*}
Then $\eta^*$ is also convex, lower semi-continuous, non-negative and satisfies $\eta^*(\f 0) = 0$.  
We introduce the subdifferential $\partial \eta$ of $\eta$ by 
\begin{align*}
\partial \eta (\f y) := \left  \{ \f z \in \Banachdual \mid 
\forall \tilde{\f y} \in\Banachdual:\ \eta( \tilde{\f y}) \leq \eta (\f y) +
\f z \cdot \np{\tilde{\f y} - \f y}
\right \} \,
\end{align*}
for $\f y\in\Banach$.
The subdifferential $\partial\eta^\ast$ of $\eta^\ast$ is defined analogously.
Then the Fenchel equivalences hold:
For $\f y,\f z\in\R^d$ we have
\begin{equation}\label{eq:fenchel}
\f z \in \partial \eta(\f y)
\quad\iff\quad 
\f y \in \partial \eta^*(\f z )
\quad\iff\quad 
\eta(\f y ) + \eta^*(\f z) = 
\f z \cdot \f y\,.
\end{equation}
A proof of this well-known result can be found in~\cite[Prop~2.33]{barbu} for example.
If $\partial\eta(\f y)$ is a singleton for some $\f y\in \R^m$, 
then $\eta$ is Fr\'echet differentiable in $\f y$ and 
$\partial\eta(\f y)=\set{D\eta(\f y)}$.
In this case, we identify $\partial\eta(\f y)$ with $D\eta(\f y)$.

 \subsection{Main result}

We introduce the notion of energy-variational solutions to 
the hyperbolic conservation law \eqref{eq}.
Consider an entropy functional $\eta : \R^m \to [0,\infty]$, $m\in\N$.
We define the total entropy functional
\begin{equation}
\label{eq:E}
\mathcal E\colon 
L^1(\T;\R^m) 
\to[0,\infty], \qquad
\mathcal{E}(\f U) = \int_{\T}\eta(\f U) \de \f x\,
\end{equation}
with domain 
$\dom\E\coloneqq \{ \f U \in L^1(\T;\R^m) \mid \E(\f U) < \infty \}$.
As the set of test functions,
we consider a closed subspace $\Y$ of $\C^1(\T;\R^m )$.
We next collect further
assumptions on $\eta$, $\f F$, and $\Y$.

\begin{hypothesis}\label{hypo}
Assume that
$\eta : \R^m \to [0,\infty]$ 
is a strictly convex and lower semi-continuous function that satisfies $\eta(\f 0)=0$
and has superlinear growth, that is,
\begin{equation}\label{eq:superlin.growth}
\lim_{ |\f y|\ra \infty} \frac{\eta (\f  y )}{|\f y|} = \infty \,. 
\end{equation}
We assume that the set 
\begin{equation}
\label{setD}
\domE:= 
\{ \f U \in \dom \E \mid \exists \{ \Phi_n\}_{n\in\N} \subset \Y : D\eta^* \circ\Phi_n \rightharpoonup \f U \text{ in } L^1(\T;\R^m)
\}
\end{equation}
is convex. 
Furthermore, let $\f F : \R^m \ra \R^{m\times d }$ 
be a measurable function 
such that
there exists a constant $C>0$ with
 \begin{equation} 
 \forall \f y \in \R^m :\quad | \f F( \f y) | \leq C(\eta(\f y)+1) \,, \label{BoundF}
 \end{equation}
and such that
\begin{equation}\label{eq:integralFentropy}
\forall \Phi\in\Y:\quad
\int_\T \f F( D\eta^\ast(\Phi(x))):\nabla\Phi(x)\,\dx=0.
\end{equation}
We further assume that there exists a convex and continuous function $\mathcal{K}: \Y \ra [0,\infty)  $
such that for any $\Phi\in\Y$ the mapping  
\begin{equation}
\domE \ra \R, \quad 
\f U \mapsto \F{\f U}{\Phi} + \mathcal{K}(\Phi) \E(\f U) \label{ass:convex}
\end{equation}
is  convex, lower semi-continuous  and non-negative.
\end{hypothesis}

Before we further explain the assumptions made in Hypothesis \ref{hypo},
let us introduce the notion of energy-variational solutions and formulate 
the main result on their existence.

\begin{definition}[Energy-variational solutions]\label{def:envar}
We call a pair $(\f U, E)  \in L^\infty (0,T;\domE)\times \BV$ 
an energy-variational solution to \eqref{eq}
if $\E (\f U) \leq E $ a.e.~on $[0,T]$,
if
\begin{equation}
\left[  E - \langle  \f U , \Phi  \rangle\right ] \Big|_{s}^t  + \int_s^t \Bb{\int_{\T}\f U \cdot \t \Phi + \f F (\f U ) : \nabla \Phi\de \f x + \mathcal{K}(\Phi) \left [\E (\f U) - E \right ] } \de \tau \leq 0 \, \label{envarform}
\end{equation}
for a.a.~$s,t\in (0,T)$, $s<t$, including $s=0$ with $\f U(0) = \f U_0$, and all $\Phi \in \C^1(  [0,T]; \Y )$,
\end{definition}

While energy-variational solutions may not have much regularity at the outset, 
we shall see that 
the initial value $\f U_0$ is attained in the weak$^\ast$ sense in $\Y^\ast$,
and that
$\f U$ and $E$ can be redefined
such that 
$E$ is non-increasing  
and $ \f U \in \C_{w^*}([0,T];\Y^*)$,
see Proposition~\ref{prop:reg} below.

As the main result of this article, we show existence of energy-variational solutions
under the previously specified assumptions. 

\begin{theorem}[Existence of energy-variational solutions]\label{thm:main}
Let Hypothesis~\ref{hypo} be satisfied, and let $\f U_0\in\domE$. Then there exists an energy-variational solution in the sense of Definition~\ref{def:envar} with $E(0+)=\mathcal{E}(\f U_0)$.
\end{theorem}

The proof of this theorem relies on a suitable time discretization and is provided
in 
Subsection \ref{subsec:existence}. 
Next we further comment on the assumptions stated in Hypothesis~\ref{hypo}
and on the solution concept of energy-variational solutions.

\begin{remark}\label{rem:integralwelldefined}
Hypothesis~\ref{hypo} ensures, 
that the integrals in \eqref{eq:integralFentropy} and \eqref{ass:convex}
are well defined.
For the integral in \eqref{ass:convex} note that 
the estimate \eqref{BoundF} implies $\snorm{\f F\circ\f U}\in L^1(\T)$ 
for all $\f U\in\domE$.
For the left-hand side of \eqref{eq:integralFentropy},
we first observe that $\partial\eta^\ast$ is single valued
by Lemma \ref{lem:convex} below,
since $\eta$ has superlinear growth.
The Fenchel equivalences \eqref{eq:fenchel} yield the identity
\[
\eta(D\eta^\ast(\Phi(x)))
=D\eta^\ast(\Phi(x)) \cdot \Phi(x)-\eta^\ast(\Phi(x)),
\]
which shows that $x\mapsto\eta(D\eta^\ast(\Phi(x)))$
is a continuous function on the compact set $\T$
and thus bounded for any $\Phi\in\Y$. Hence $ D\eta^* \circ \Phi \in \dom \E$.
Therefore, 
inequality \eqref{BoundF}
yields a bound for the integrand in \eqref{eq:integralFentropy}.
\end{remark}

\begin{remark}
\label{rem:domE}
The convexity assumption on $\domE$ can be seen as a compatibility condition on the space $\Y$ and the entropy $\eta$. 
We note that  $D\eta^*\circ\Phi \in \dom\E$ for $\Phi \in \Y$ as shown in Remark~\ref{rem:integralwelldefined}. 
Moreover, for any sequence $ \{ \f U_n\}_{n\in\N} \subset \domE$ with bounded entropies, $ \E(\f U_n) \leq C $, there is a convergent subsequence with limit $\f U \in \domE$. Indeed, \eqref{eq:superlin.growth} yields the existence of a  subsequence weakly converging to $\f U$ in $L^1(\T;\R^m)$ with $ \E(\f U) \leq C$, see Lemma~\ref{lem:delavalle} below. A diagonalization argument gives a sequence $\{ \Phi_n\}_{n\in\N}\subset\Y$ with $ D\eta^*\circ\Phi_n \rightharpoonup \f U$ in $L^1(\T;\R^m)$, which shows $\f U \in \domE$.

In the case of a quadratic functional $\eta(\f y)=a\snorm{\f y}^2$, $a>0$, 
the set $\domE$ is the weak closure of $\Y$ in $L^1(\T;\R^m)$.
Since $\Y$ is a linear subspace and $\eta$ is quadratic,
this is nothing else than the strong closure of $\Y$ in 
$L^2(\T;\R^m)$.
In particular, the convexity of $\domE$ is satisfied trivially.

In the case $\Y=\C^1(\T;\R^m)$, 
we have $\domE=\dom\E$.
In particular, $\domE$ is convex.
Since $\dom(\partial\E)$ is dense in $\dom\E$ 
(see \cite[Corollary 2.44]{barbu})
this follows from the above approximation property
and $\dom(\partial\E)\subset\domE$.
To see the latter, let 
$\f U\in\dom(\partial\E)$.
From~\cite[Prop.~2.53]{barbu}, we infer that 
the existence of $\Phi \in L^\infty( \T;\R^m)$ such that 
$\Phi ( \f x ) \in \partial \eta(\f U(\f x ))$ for a.a.~$x\in\T$,
that is,
$D\eta^*(\Phi ( \f x )) = \f U(\f x )$
by the Fenchel equivalences~\eqref{eq:fenchel}.
The density of $\C^1(\T;\R^m) $  in $L^\infty( \T;\R^m)$ with respect to the weak$^*$ topology, guarantees the existence of a sequence $\{ \Phi _n\}_{n\in\N} \subset \C^1(\T;\R^m) $ with $ \| \Phi_n \|_{L^\infty(\T;\R^m)} \leq \| \Phi \|_{L^\infty(\T;\R^m)} $ and $ \Phi_n \ra \Phi$ a.e.~in $\T$, see~\cite[Ex.~4.25]{brezis}. Lebesgue's convergence theorem allows to conclude that $D\eta^*(\Phi_n) \ra \f U $ in $L^1(\T;\R^m)$ by the continuity of $D\eta^*$,
which shows $\f U\in\domE$.
\end{remark}

\begin{remark}
Instead of assuming that $\eta(\f 0)=0$ and
$ \eta\geq 0$,
we may consider a function $\eta : \R^m \to (-\infty,\infty]$
that attains its minimum at $\f 0$.
Indeed, the original assumptions
can then be recovered by simply adding a suitable constant to $\eta$.
\end{remark}

\begin{remark}\label{rem:integralcondition}
Equation \eqref{eq:integralFentropy} ensures that the 
total entropy is conserved along smooth solutions.
Indeed, if $\f U$ is a solution 
and all functions are sufficiently smooth,
then we formally have
\[
\ddt\E(\f U)
=\!\!\int_\T \partial_t \f U\cdot D\eta(\f U) \,\dx
=-\!\!\int_\T [\dv \f F(\f U)]\cdot D\eta(\f U) \,\dx
=  \!\!\int_\T \f F( \f U):\nabla D\eta(\f U)\,\dx
= 0,
\]
where the last identity follows from 
\eqref{eq:integralFentropy} with 
$\Phi=D\eta(\f U)$.
Classically, 
this conservation property is ensured by 
requiring the existence of an
entropy flux $\f q: \R^ m \ra \R^d$ such that 
\begin{equation}
\label{eq:entropypair}
 D \eta (\f y) ^T D\f  F (\f y )= D \f q (\f y)^\transpose
\end{equation}
for all $\f y\in\R^m$,
which is a shorthand for the relation
\[
 D \eta (\f y) ^T D \f F_j (\f y )= D \f q_j(\f y)^\transpose \qquad (j=1,\ldots,d).
\] 
Clearly, this identity only makes sense if $\eta$ and, in particular,
$\f F$ are smooth enough.
This smoothness cannot be guaranteed for general conservation laws
as we shall see in Section \ref{sec:comprEuler}
in the context of the compressible Euler equations.
However, if this is the case,
then \eqref{eq:integralFentropy} follows from \eqref{eq:entropypair}.
Indeed, setting $\f U=D\eta^\ast(\Phi)$, that is, $\Phi=D\eta(\f U)$,
and integrating by parts,
we deduce
\[
\begin{aligned}
\int_\T \f F( D\eta^\ast(\Phi)):\nabla\Phi\,\dx
&=\int_\T \f F( \f U):\nabla D\eta(\f U)\,\dx
=-\int_\T \bb{D\eta(\f U)^T D\f F( \f U)}:\nabla \f U\,\dx
\\
&=-\int_\T D \f q(\f U):\nabla \f U\,\dx
=-\int_\T \dv \f q(\f U)\,\dx
= 0 \,.
\end{aligned}
\]

Instead of verifying \eqref{eq:integralFentropy} directly,
one can also show existence of a vector field $\tilde{\f q}\colon\R^m\to\R^d$
such that $\tilde{\f q}\circ D\eta^\ast\in\C^1(\R^m;\R^d)$ and
\begin{equation}\label{eq:entropyflux.new}
\forall \f z\in\R^m:\quad
\f F(D\eta^\ast(\f z))= D\bb{\tilde{\f q}\circ D\eta^\ast}(\f z).
\end{equation}
This implies
\[
\f F(D\eta^\ast(\Phi))\nabla\Phi= \dv \bb{\tilde{\f q}(D\eta^\ast(\Phi))}
\]
for all $\Phi\in\C^1(\T;\R^m)$, so that
\eqref{eq:integralFentropy} follows from the divergence theorem.
Observe that, in contrast to \eqref{eq:entropypair},
condition \eqref{eq:entropyflux.new} does not require 
$\f F$ to be differentiable. 
Moreover, we do not require differentiability of 
$\tilde{\f q}$ and $D\eta^\ast$ 
but merely of their composition.
This distinction can be helpful since
there are standard cases where
$\eta^\ast$ is not twice differentiable,
for example, the compressible Euler equations, 
which we study in Section~\ref{sec:comprEuler}.

Formally, the relations~\eqref{eq:entropyflux.new} and~\eqref{eq:entropypair} 
are equivalent in the case that $\eta^* \in \C^2(\R^m)$
and $D^2\eta^\ast(\f z)$ is invertible at each $\f z\in \R^m$. 
Indeed, choosing $\f y = D \eta^*(\f z)$, we find by~\eqref{eq:fenchel} and the chain rule that
\[
\begin{aligned}
\bb{D\eta(D\eta^*(\f z))^T & D \f F(D\eta^*(\f z))  - D \f q(D \eta^*(\f z)) } 
D^2\eta^*(\f z )
\\
 &=\f z ^T D\nb{ \f F(D\eta^*(\f z)) } - D\nb{ \f q(D \eta^*(\f z)) }
 \\
 &= D \bb{\f z ^T \f F(D\eta^*(\f z)) - \f q(D \eta^*(\f z))}
 -\f F( D \eta^*(\f z)) \,.
\end{aligned}
\]
Hence, 
\eqref{eq:entropypair} is satisfied
if and only if \eqref{eq:entropyflux.new} holds for
$\tilde{\f q}(\f U) = D\eta(\f U)^T \f F (\f U) - \f q(\f U)$. 
\end{remark}

\begin{remark}
In case that $\f F$ is entropy-convex, \textit{i.e.}, there exists a constant $\lambda > 0$ such that 
$|\f F| + \lambda \eta$ is a convex, weakly lower semi-continuous function on $\R^m$, 
we may choose $\mathcal{K}(\Phi)= \lambda\| \nabla \Phi\|_{L^\infty(\T)}$.
We shall use a similar functional $\mathcal K$ in Subsection \ref{sec:magneto},
but finer choices may be possible as we shall see 
in Subsection~\ref{sec:incompEuler} and Section~\ref{sec:comprEuler}.
\end{remark}

\begin{remark}[Boundary conditions]
In order to simplify the analysis, we restrict ourselves to the case of periodic boundary conditions. 
But the method can also be adapted to more general boundary conditions.  These can usually be included into our framework by modification of the space of test functions $\mathbb Y$;
see also Remark~\ref{rem:boundary} below.
\end{remark}

\subsection{Auxiliary results}

Before we start with the analysis of energy-variational solutions, 
we prepare several auxiliary lemmas.
We start with the following basic result 
on an affine linear variational inequality.

\begin{lemma}\label{lem:var.affine}
Let $\mathbb X$ be a Banach space,
and let $a_1, a_2 \in \R$ and $y_1,y_2\in \mathbb X^\ast$ such that
\[
a_1 + \langle  y_1, x \rangle \leq a_2 + \langle  y_2, x \rangle
\]
for all $x\in \mathbb X$. 
Then $a_1\leq a_2$ and $y_1=y_2$.
\end{lemma}

\begin{proof}
The choice $x=0$ directly yields $a_1\leq a_2$.
To infer $y_1=y_2$, let $\bar x\in\X$ and $\lambda >0$.
Choosing $x=\lambda \bar x$ and dividing by $\lambda$, we deduce
 \[
\lambda^{-1}a_1 + \langle y_1, \bar x \rangle 
\leq \lambda^{-1}a_2 + \langle y_2,\bar x\rangle.
\]
A a passage to the limit $\lambda\to\infty$ yields 
$\langle y_1,\bar x\rangle \leq \langle y_2,\bar x\rangle$.
Choosing $x=-\lambda \bar x$ and proceeding in the same way
results in the converse inequality, and we obtain
$\langle y_1,\bar x\rangle = \langle y_2,\bar x\rangle$.
Since $\bar x\in \mathbb X$ was arbitrary,
this yields $y_1=y_2$ and completes the proof.
\end{proof}

The next result yields the equivalence of a
pointwise inequality and its variational formulation.

\begin{lemma}\label{lem:invar}
Let $f\in L^1(0,T)$, $g\in L^\infty(0,T)$ and $g_0\in\R$.
Then the following two statements are equivalent:
\begin{enumerate}[label=\roman*.]
\item
The inequality 
\begin{equation}
-\int_0^T \phi'(\tau) g(\tau) \de \tau  + \int_0^T \phi(\tau) f(\tau) \de \tau - \phi(0)g_0 \leq 0 
\label{ineq1}
\end{equation}
holds for all $\phi \in {\C}^1_c ([0,T))$ with $\phi \geq 0$.
\item
The inequality
\begin{equation}
    g(t) -g(s) + \int_s^t f(\tau) \de \tau \leq 0 
    \label{ineq2}
\end{equation}
holds for a.e.~$s,\, t\in[0,T)$ with $s<t$,
including $s=0$ if we replace $g(0)$ with $g_0$.
\end{enumerate}
If one of these conditions is satisfied, 
then $g$ can be identified with a function in $\BV$ 
such that
\begin{equation}
    g(t+) -g(s-) + \int_s^t f(\tau) \de \tau \leq 0 \,
    \label{ineq.pw}
\end{equation}
for all $s,t\in[0,T)$ with $s\leq t$,
where we set $g(0-)\coloneqq g_0$.
In particular, it holds $g(0+)\leq g_0$ and $g(t+)\leq g(t-)$ for all $t\in(0,T)$.
\end{lemma}

\begin{proof}
To see that \eqref{ineq1} implies \eqref{ineq2},
one can use a standard procedure and
approximate the indicator function of the interval $(s,t)$ by 
elements of $\C^1_c([0,T))$.
For the inverse implication,
first note that \eqref{ineq2} implies that $g$ coincides a.e.~with an element of $\BV$.
Hence, one-sided limits of $g$ exist in each point, 
and we deduce \eqref{ineq.pw} from \eqref{ineq2}.
The choice $s=t$ in \eqref{ineq2} implies $g(t+)\leq g(t-)$
and $g(0+)\leq g_0$.
Now let $0\leq\phi \in {\C}^1_c ([0,T))$ 
and consider a partition $0=s_0\leq t_0< s_1<t_1<\dots< s_N<t_N<T$ of $[0,T]$ 
such that 
\[
\phi' \geq 0 \quad\text{in } [t_{j-1},s_j]\,,
\qquad
\phi' \leq 0 \quad\text{in } [s_j,t_j]\,,
\qquad 
\phi=\phi'=0 \quad \text{in } [t_N,T]\,.
\]
To show \eqref{ineq1}, 
we subdivide the left-hand side of this inequality accordingly.
Since $\phi'\leq0$ in $[s_j,t_j]$,
we can use \eqref{ineq.pw} with $s=s_j$ and integration by parts
to estimate
\[
\begin{aligned}
&-\int_{s_j}^{t_j}\phi'(\tau) g(\tau)\de\tau
\leq 
-\int_{s_j}^{t_j}\phi'(\tau) \Bp{g(s_j-)-\int_{s_j}^{\tau}f(r)\,\de r}\de\tau
\\
&\qquad
= -\phi(t_j)\Bp{g(s_j-)-\int_{s_j}^{t_j}f(r)\,\de r}+\phi(s_j) g(s_j-)
-\int_{s_j}^{t_j}\phi(\tau) f(\tau)\de\tau,
\end{aligned}
\] 
where for $j=0$ we have to replace $g(s_0-)$ with $g_0$.
Since $\phi'\geq0$ in $[t_{j-1},s_j]$, 
we can use \eqref{ineq.pw} with $t=s_j$ in a similar way
to conclude
\[
\begin{aligned}
&-\int_{t_{j-1}}^{s_j}\phi'(\tau) g(\tau)\,\de\tau
\leq 
-\int_{t_{j-1}}^{s_j}\phi'(\tau) \Bp{g(s_j+)+\int_{\tau}^{s_j}f(r)\de r}\de\tau
\\
&\qquad
= -\phi(s_j) g(s_j+)
+\phi(t_{j-1})\Bp{g(s_j+)+\int_{t_{j-1}}^{s_j}f(r)\de r}
-\int_{t_{j-1}}^{s_j}\phi(\tau) f(\tau)\de\tau.
\end{aligned}
\] 
Summing up and using $\phi=\phi'=0$ in $[t_N,T]$, we obtain
\[
\begin{aligned}
-\int_0^T & \phi'(\tau) g(\tau)\,\de\tau+\int_0^T \phi(\tau) f(\tau) \de \tau 
-\phi(0)g_0
\\
&=-\sum_{j=0}^N \int_{s_j}^{t_j}\phi'(\tau) g(\tau)\de\tau
-\sum_{j=1}^N\int_{t_j-1}^{s_j}\phi'(\tau) g(\tau)\de\tau
+\int_0^T \phi(\tau) f(\tau) \de \tau -\phi(0)g_0
\\
&\leq
\sum_{j=1}^N \phi(s_j)\bp{g(s_j-)-g(s_j+)}
+\sum_{j=1}^N \phi(t_j) \Bp{g(s_{j+1}+)-g(s_j-)+\int_{s_j}^{s_{j+1}}f(r)\de r}
\end{aligned}
\]
Since $\phi\geq 0$, 
invoking inequality~\eqref{ineq.pw} 
and that $g(t +)\geq g(t -)$,
we can estimate the terms in the last line by $0$
and finally conclude~\eqref{ineq1}.
\end{proof}

Next we show an adaption of a well-known 
theorem by de la Vall\'ee Poussin, see 
\cite[Sect.~1.2, Theorem 2]{RaoRen_OrliczSpaces_1991} for example.
For the sake of completeness, we give a proof here.
Observe that the statement remains valid if 
$\T$ is replaced with any other finite measure space. 

\begin{lemma}\label{lem:delavalle} 
Let $\psi : \R^m \ra [0,\infty]$ have superlinear growth, \textit{i.e.},
$\lim_{| \f y| \ra \infty} \psi( \f y) / | \f y| = \infty $,
and let $\mathcal{F}\subset L^1(\T;\R^m) $ and
$C>0$ such that 
\[
\forall\, \f U \in \mathcal{F}: \quad
\int_\T \psi(\f U ) \de \f x \leq C \,.
\]
Then the set $\mathcal{F}$ is equi-integrable and therewith 
relatively weakly compact in $L^1(\T; \R^m)$. 
\end{lemma}
\begin{proof}
Let $\varepsilon>0$ and set $M=2C/\varepsilon$.
By assumption, we can choose $R>0$ so large that $\snorm{\f y}>R$ 
implies $\psi(\f y)> M\snorm{\f y}$.
Let $A\subset\T$ be a measurable set
with $\snorm{A}<\frac{\varepsilon}{2R}$.
Then
\[
\int_A\! |\f U| \,\dx 
= \int_{\setc{\f x\in A}{\,\snorm{\f U(\f x)}\leq R}}\!|\f U| \,\dx 
+ \int_{\setc{\f x\in A}{\,\snorm{\f U(\f x)}> R}}\!|\f U| \,\dx
\leq R \snorm{A} 
+ \frac{1}{M}\int_{\T}\!\psi(\f U(\f x)) \,\dx
\leq \varepsilon.
\]
This shows
\[
\lim_{\snorm{A}\to 0} \sup_{\f U\in\mathcal F} \int_A |\f U| \,\dx =0,
\]
that is, the equi-integrability of $\mathcal F$. 
The relative weak compactness of $\mathcal F$ 
now follows from the Dunford--Pattis theorem~\cite[Thm.~3.2.1]{dunford}. 
\end{proof}

The next lemma collects 
useful properties 
of a convex functionals with superlinear growth.

\begin{lemma}\label{lem:convex}
Let
$\eta :\Banach  \ra [0,\infty]$ be a strictly convex, lower semi-continuous function with $\eta (\f 0 )= 0$ and
\eqref{eq:superlin.growth}.
Then the set-valued operator 
$\partial \eta : \Banach \ra \Banachdual$  is maximal monotone and surjective.
Moreover, the convex conjugate $\eta^\ast$ is globally defined and continuously differentiable.
In particular,
\[
\forall \f z\in\R^d: \quad
(\partial\eta)^{-1}(\set{\f z})=\partial\eta^\ast(\f z)=\set{D\eta^\ast(\f z)}.
\]
\end{lemma}

\begin{proof}
The subdifferential $\partial\eta$ induces a maximal monotone operator according to~\cite[Thm.~2.43]{barbu}, and from~\cite[Prop.~2.47]{barbu} we infer that this operator is surjective. 
The Fenchel equivalences \eqref{eq:fenchel} allow to identify this inverse with the subdifferential of the conjugate $\eta^*$. 
Note that $\eta^*$ is even Gateaux-differentiable~\cite[Rem.~2.41 and Prop.~2.40]{barbu}
and continuous with $\dom\eta^\ast=\R^d$~\cite[Prop.~2.25 and Thm.~2.14]{barbu}.
The  assertion that $\partial \eta^*$ is single-valued and continuous can be found in~\cite[Thm.~5.20]{roubicek}.
\end{proof}

We use some of these properties 
to prove the following lemma that shows a way how to 
continuously interpolate between $0$ and a given value in the range of $\E$
defined in~\eqref{eq:E}.

\begin{lemma}\label{lem:surjective} 
In the situation of Lemma \ref{lem:convex}, 
let $\Phi \in\C(\T;\R^m)$ and $\tU= D\eta^\ast\circ\Phi$.
Then the mapping
\[
    \mathcal G\colon [0,1]\to [0,\E(\tU)], \quad
    \alpha \mapsto
     \E ( D \eta^* (\alpha \Phi)) 
\]
is well defined, continuous and surjective.
\end{lemma}

\begin{proof}
Fix $\f x\in\T$, and let 
$\f y =\tilde{\f U}(\f x)\in \dom \partial\eta$ 
and $\f z=\Phi(\f x)\in\partial\eta(\f y)$.
Consider 
\[
     f\colon [0,1]\to [0,\infty], \quad
    \alpha \mapsto
    \eta  ( D \eta^* (\alpha \f z)) 
\]
Since $\partial \eta^\ast$ has full domain and is 
single valued according to Lemma~\ref{lem:convex}, 
the mapping is well defined. 
Via the Fenchel equivalences \eqref{eq:fenchel}, we may further express $f$ as
\[
f(\alpha) = \eta( D\eta^* (\alpha \f z)) = \langle  D \eta^* (\alpha \f z ) , \alpha \f z \rangle  - \eta^* ( \alpha \f z) .
\] 
This shows that $f(\alpha)$ is finite
and that $f$ is continuous 
since $\eta^*$ and $D\eta^* $ are continuous by Lemma~\ref{lem:convex}.
Moreover, $f(0)=0$ and $f(1)=\eta(\f y)$,
and via Fenchel's identity and the monotonicity of $D \eta^*$, 
we further observe for $0 \leq  \beta < \alpha \leq 1 $ that
\[
\begin{aligned}
   f(\alpha )-f(\beta) &={} 
\langle    D\eta^*(\alpha \f z) , \alpha \f z\rangle   -  \langle D\eta^*(\beta \f z) \,\beta \f z\rangle  -\bp{\eta^*(\alpha \f z)- \eta^*(\beta \f z)} \\
   &\geq \langle  D\eta^*(\alpha \f z) ,\alpha \f z\rangle   -  \langle D\eta^*(\beta \f z) ,\beta \f z\rangle + \langle D\eta^* (\alpha \f z),\beta \f z - \alpha \f z\rangle  
   \\
   &=   \frac{\beta}{\alpha - \beta } \left \langle D\eta^*(\alpha \f z)-  D\eta^*(\beta \f z) ,\alpha \f z - \beta \f z \right \rangle \geq 0 \,.
\end{aligned}
\]
Hence, $f$ is a continuous and non-decreasing mapping with range $[0,\eta(\f y)]$.
This implies that the mapping $\mathcal G$ is well defined with
$0=\mathcal G(0)\leq\mathcal G(\alpha)\leq \mathcal G(1)=\mathcal E(\tU)$
for all $\alpha\in[0,1]$.
Using Lebesgue's theorem on dominated convergence,
we further conclude that $\mathcal G$ is continuous,
which also implies that $\mathcal G$ is surjective.
\end{proof}

We shall also make use of the following result on the extension
of certain linear functionals.
\begin{lemma}\label{lem:hahn}
Let  $ \f l : \mathcal{V} \ra \R$ be a linear continuous functional,
where $ \mathcal{V} $ is a closed subspace of 
\[
\mathcal U:= \setcl{\f \varphi\in \C_0^1(\O\times[0,T);\R^d)}{ \int_\O\f\varphi\de x =0}.
\]
Set 
\[
\mathcal I \colon \mathcal U \to L^1(0,T; \C(\O; \R_{\mathrm{sym}}^{d\times d })),\qquad
\mathcal I(\f \psi)=(\nabla\f\psi)_{\mathrm{sym}},
\]
and let
$\mathfrak p : L^1(0,T;\C(\O ; \R^{d\times d}_{\sym}))  \ra \R$ be a sublinear mapping such that 
\begin{equation}
    \forall \f \psi\in \mathcal V \colon \quad \langle \f l , \f \psi \rangle \leq \mathfrak p ( \mathcal{I}(\f\psi ) ) \,.\label{est:l}
\end{equation}
Then there exists
 an element
$$\mathfrak R \in (L^1(0,T; \C(\O; \R_{\sym}^{d\times d }
)))^*=L^\infty_{w^*}(0,T;\mathcal{M}(\O ; \R_{\sym}^{d\times d}
)) $$ 
satisfying 
\[
\forall\Phi\in L^1(0,T; \C(\O; \R_{\sym}^{d\times d })):\ \langle -\mathfrak R, \Phi\rangle \leq \mathfrak p(\Phi),
\qquad
\forall\f \psi\in \mathcal{V}: \ 
\langle -\mathfrak R, \mathcal I(\f\psi)\rangle = \langle \f l, \f \psi\rangle.
\]
\end{lemma}

\begin{proof}
First consider $\f \psi\in\mathcal V$ with $\mathcal I(\f \psi)=0$.
This implies that $\f \psi(\cdot,t)$ is affine linear,
and since $\f \psi\in\mathcal V$ is spatially periodic and has vanishing mean value,
this is only possible for $\f \psi=0$. 
Therefore, $\mathcal I$ is injective,
and on its image $\mathcal W=\mathcal I(\mathcal V)$
we can define the functional $L$ by $\langle L,\Psi\rangle=\langle \f l,\f \psi\rangle$
for $\Psi=\mathcal I(\f \psi)\in\mathcal W$.
Then estimate \eqref{est:l} implies
\begin{equation}\label{linearformest}
\langle L, \Psi \rangle \leq \mathfrak p(\Psi)
\end{equation}
for all $\Psi\in\mathcal W\subset L^1(0,T;\C(\O ; \R^{d\times d} _{\sym}))$.
By the Hahn--Banach theorem (see e.g.~\cite[Thm~1.1]{brezis}), we may extend $L$ 
from $\mathcal W$ to a linear functional on $L^1(0,T;\C(\O ; \R^{d\times d} _{\sym}))$. 
Using the Riesz representation theorem, 
we may identify this extension with an object $-\mathfrak R$
such that the asserted properties are satisfied.
\end{proof}

\section{Properties and existence of energy-variational solutions}
\label{sec:hyper}

In this section we collect several general properties 
of energy-variational solutions that follow directly from Definition~\ref{def:envar}.
Moreover, under additional regularity assumptions, we can show 
a relative entropy inequality, which yields a weak-strong uniqueness principle.
Finally, in Subsection~\ref{subsec:existence},
we introduce a time-discrete scheme that leads to 
the existence of energy-variational solutions as claimed in 
Theorem \ref{thm:main}.

\subsection{General properties}

Let us begin with some continuity properties of 
energy-variational solutions, 
which follow directly from  Definition~\ref{def:envar}.

\begin{proposition}\label{prop:reg}
Let $(\f U, E)$ be an energy-variational solution in the sense of Definition~\ref{def:envar}. 
Then $\f U$ and $E$ can be redefined on a subset of $[0,T]$ of measure zero 
such that $E$ is a non-increasing function 
and such that $ \f U \in \C_{w^*}([0,T];\Y^*)$
with $\f U(0)=\f U_0$ in $\Y^\ast$.
Then inequality~\eqref{envarform} is fulfilled everywhere in $[0,T]$ in the sense that
for all $\Phi \in \C^1(  [0,T]; \Y )$ it holds
\begin{equation}
\left[  E - \langle \f U ,  \Phi  \rangle\right ] \Big|_{s-}^{t+}  + \int_s^t \int_{\T}\f U \cdot  \t \Phi  + \f F (\f U ) : \nabla \Phi+ \mathcal{K}(\Phi) \left [\E (\f U) - E \right ]  \de \f x \de \tau \leq 0 \, \label{envarform2}
\end{equation}
for all $s\leq t \in [0,T)$,
where $E(0-)-\langle  \f U(0-), \Phi(0-)  \rangle
\coloneqq E(0+)-\langle  \f U_0, \Phi(0)  \rangle$.
\end{proposition}

\begin{proof}
Setting $\Phi\equiv 0$ in inequality~\eqref{envarform}, we infer 
that $ E\big|^t_s \leq 0$ for a.e.~$t>s\in(0,T)$. 
Since $E \in \BV$,  all left-sided and right-sided limits exist and $E$ is continuous 
except for countably many points,
so that we can redefine $E$ such that it is non-increasing. 
For any fixed $\Phi \in \C^1(  [0,T]; \Y )$ we further observe that 
\[
\begin{aligned}
\left[E-   \langle \f U ,  \Phi  \rangle\right ] \Big|_{s}^t  \leq{}& - \int_s^t \int_{\T}\f U \cdot \t \Phi  + \f F (\f U ) : \nabla \Phi+  \mathcal{K}(\Phi) \left [\E (\f U)-E  \right ]  \de \f x \de s 
\\
\leq {}& \int_s^t \Bb{\int_{\T}  \eta(\f U)+\eta^*(\t \Phi)  \de \f x + \mathcal{K}(\Phi) E }\de \tau 
 \,
\end{aligned}
\]
for a.e.~$t>s \in (0,T)$, where we used 
the Fenchel--Young inequality
and the non-negativity of the function in~\eqref{ass:convex}. 
This implies that $t \mapsto E(t)-\langle  \f U(t),\Phi(t) \rangle \in \BV$. 
In particular, left-sided and right-sided limits of this function exist,
and passing to those limits in \eqref{envarform} yields \eqref{envarform2}.
Choosing now $s=t$ and $\Phi\in\Y$ independent of time, we infer that 
\[ 
\left[  E - \langle  \f U , \Phi  \rangle\right ] \Big|_{t-}^{t+}  \leq 0 \quad \text{for all }t\in(0,T) \text{ and } \Phi \in \Y\,.
\]
Lemma \ref{lem:var.affine} now yields $\f U(t+) = \f U(t-)$ in $\mathbb{Y}^\ast$
 for all $t\in (0,T)$ \textit{i.e.,}
we can redefine $\f U$ on a set of measure $0$ such that
$ \f U \in \C_{w^*}([0,T];\Y^*)$. 
\end{proof}

\begin{proposition}
\label{prop:betterreg}
Assume that for two elements $ \f V $, $\f W \in \domE $ with $\langle \f V - \f W , \Phi \rangle
 = 0 $ for all $ \Phi \in \Y$  it holds $ \f V = \f W$. 
 Then we have $\f U \in \C_{w}([0,T]; L^1(\T;\R^m))$. Furthermore, if   $\E(\f U_0)=E(0)$,  the initial value is  attained in the strong sense in $L^1(\T;\R^m)$. 
\end{proposition}
\begin{proof}
Let $t\in [0,T]$ and consider a sequence $ \{ t_n\}_{n\in\N} \subset [0,T]$ with $t_n \ra t$.
Then
$ \E(\f U(t_n)) \leq E(t_n)\leq E_0$ for $n\in\N$,
and from \eqref{eq:superlin.growth}
and Lemma~\ref{lem:delavalle} we infer that the set $ \{ \f U(t_n)\}_{n\in\N} $ is relatively weakly compact in $L^1(\T;\R^m)$.
Hence, we may extract a subsequence such that 
\begin{equation*}
 \f U(t_{n_k})\rightharpoonup \f A _t \quad\text{in }L^1(\T;\R^m)\,
\end{equation*}
for some $\f A_t\in \domE $.
As shown above, we also have
$$ \f U(t_{n_k}) \stackrel{*}{\rightharpoonup} \f U(t)  \quad \text{in }\Y^*\,.$$  
We infer that $ \langle \f U (t) , \Phi \rangle = \langle \f A_t , \Phi \rangle $ for all $ \Phi \in \Y$.
The assumption implies
 $\f U(t)=\f A_t$.  
Due to the uniqueness of the weak limit, all subsequences converge to this limit, 
so that $ \f U \in \C_w([0,T];L^1(\T;\R^m))$. 

Moreover, if   $\E(\f U_0)=E(0)$,   we infer 
\[ 
E(0) \geq \lim _{t\searrow 0 } E( t ) \geq \lim_{t \searrow 0} \E(\f U(t)) \geq \E(\f U_0) = E(0) \,
\]
due to the monotonicity of the function $E$ and the weak lower semi-continuity of $\E$.
We conclude that $ \E(\f U(t))  \ra \E(\f U_0) $ as $t \ra 0$. 
Since we also have $\f U (t) \rightharpoonup  \f U_0$,
from the strict convexity of $\E$, we infer that $\f U (t) \to   \f U_0$ strongly in $L^1(\T ; \R^m) $  by~\cite[Thm.~10.20]{singular}. 
\end{proof}

\begin{remark}[Semi-flow property]
We note that energy-variational solutions fulfill the semi-flow property.
This means that 
the restriction of a solution to a smaller time interval
as well as the concatenation
of two solutions $(\f U_1, E_1)$ and $(\f U_2, E_2)$
on subsequent time intervals $(t_0,t_1)$ and $(t_1,t_2)$ 
with $(\f U_1(t_1-), E_1(t_1-))=(\f U_2(t_1+), E_2(t_1+))$.
is again a solution.
This follows from Proposition~\ref{prop:reg} due to inequality~\eqref{envarform2} 
for all $t\geq s\in[0,T]$
and the weak$^\ast$ continuity of the solution.
\end{remark}

\begin{proposition}[Solution set]
The set of all energy-variational solutions with common initial value 
$\f U_0\in\mathbb{V}$ is convex.
Moreover, let $\E( \f U_0)\leq B$ for some $B>0$, and
let $\mathcal S$ be the set of all energy-variational solutions $\np{\f U, E}$
with initial value $\f U_0\in\domE$ and $E(0)\leq B$.
Then $\mathcal S$ is compact in  $L^\infty(0,T;L^1(\T)) \times \BV$
with respect to the weak$^*$ topology in $\BV$ and 
the weak$(^\ast)$ topology in $L^\infty(0,T;L^1(\T))$
defined in \eqref{eq:weakconv.LinfL1}.
\end{proposition}

\begin{proof}
Using the convexity of $\E$ and of the mapping from \eqref{ass:convex},
one readily sees that 
all terms involving $(\f U, E)$ appear in a convex way
in \eqref{envarform}. 
Therefore, the convex combination of two energy-variational solutions
with coincident initial value
is again an energy-variational solution with the same initial value.

Now consider the set $\mathcal S$. 
By Proposition \ref{prop:reg}, 
we may assume that for all $(\f U,E)\in\mathcal S$ the function $E$ is non-increasing, 
which implies that $ | E | _{\text{TV}([0,T])} \leq B$. 
Due to the inequality $\E(\f U(t)) \leq E(t)$ for a.a.~$t\in [0,T]$ 
and
the superlinear growth of $\eta$,
we infer from Lemma~\ref{lem:delavalle} 
and Helly's selection theorem (cf.~\cite[Thm.~1.126]{barbu}) that 
any sequence in $\mathcal S$ contains a subsequence 
$\{ (\f U^n , E^n)\}_{n\in \N}$
 such that
\begin{equation}
\begin{aligned}
{\f U}^n &\xrightharpoonup{(\ast)} \f U &&\text{in } L^\infty (0,T; L^1 (\T))\,, \\
{E}^n &\xrightharpoonup{\phantom{(}\ast\phantom{)}} {E} && \text{in }\BV\,,\\
{E}^n(t) &\xrightarrow{\phantom{(\ast)}} {E}(t) && \text{for all }t\in[0,T]\,.
\label{weakconv}
\end{aligned}
\end{equation}
For the initial values, we may further extract a subsequence such that $E^n(0+) \ra E_0$ 
for some $E_0\leq B$, and we have
$\f U^n(0) = \f U_0$ for all $n\in \N$. 
Using Lemma~\ref{lem:invar}, 
we may rewrite the energy-variational inequality \eqref{envarform}
in its weak form
\begin{align*}
- \int_0^T \phi' \left [ E ^n - \langle \f U^n, \Phi \rangle \right ]\de t &- \phi(0) \left [ E ^n (0+)- \langle  \f U_0 ,\Phi(0)\rangle \right ]\\& + \int_0^T \phi \left [\F{\f U^n}{\Phi} + \mathcal{K}(\Phi )\left [\E(\f U^n) - E^n \right ] \right ] \de t\leq 0 
\end{align*}
for all $\phi \in \C_c^1( [0,T)) $ with $\phi \geq 0$ and for all $\Phi \in \C^1(  [0,T]; \Y )$. 
Via the convergences~\eqref{weakconv}, we may pass to the limit in this formulation 
and obtain, again by Lemma~\ref{lem:invar}, the formulation~\eqref{envarform}. 
Moreover, the weak lower semi-continuity of $\E$ allows to deduce that $ E (t) \geq \E(\f U(t))$ for a.e.~$t\in(0,T)$.
Consequently, $\np{\f U, E}$ is an energy-variational solution in $\mathcal S$.
\end{proof}

\begin{proposition}\label{prop:equality}
 Let $(\f U, E) \in L^\infty(0,T;\domE) \times \BV$ be an energy-variational solution in the sense of Definition~\ref{def:envar}, and let the regularity weight $\mathcal{K}$ be homogeneous of degree one, \textit{i.e.,} $ \mathcal{K}(\alpha \Phi ) = \alpha \mathcal{K}(\Phi)$ for all $\alpha \in [0,\infty)$ and $\Phi \in \Y$. 
 Then the inequality~\eqref{envarform} is equivalent to the two inequalities 
 \begin{equation}\label{twoineq}
 E\Big|_s^t \leq 0, \qquad
 - \langle  \f U , \Phi  \rangle \Big|_{s}^t  + \int_s^t \!\!\int_{\T}\f U \cdot \t \Phi  + \f F (\f U ) : \nabla \Phi\de \f x + \mathcal{K}(\Phi) \left [\E (\f U) - E \right ]  \de \tau \leq 0
 \end{equation}
 for a.a.~$s,t\in (0,T)$, $s<t$, and for all $\Phi \in \C^1(  [0,T]; \Y )$.
\end{proposition}

\begin{proof}
Summation of the two inequalities in~\eqref{twoineq} directly gives the inequality~\eqref{envarform}.
For the converse direction, the first inequality in~\eqref{twoineq} can be deduced from~\eqref{envarform} by choosing $\Phi \equiv 0$. 
In order to infer the second inequality in~\eqref{twoineq}, we choose $\Phi = \alpha \Psi$ in~\eqref{envarform} for $\alpha >0$ and $\Psi \in \C^1([0,T];\Y)$. Multiplying the resulting inequality by $\frac{1}{\alpha}$ implies 
\begin{align*}
\left[ \frac{1}{\alpha}  E - \langle \f U ,\Psi  \rangle\right ] \Big|_{s}^t  + \int_s^t \int_{\T}\f U \cdot \t \Psi  + \f F (\f U ) : \nabla \Psi\de \f x + \mathcal{K}(\Psi) \left [\E (\f U) - E \right ]  \de \tau \leq 0\,.
\end{align*}
Passing to the limit $\alpha \ra \infty$, we infer the second inequality in~\eqref{twoineq}. 
\end{proof}

\subsection{Relative entropy and weak-strong uniqueness}

In order to derive a relative entropy inequality
for energy-variational solutions, we make the following assumptions
on higher regularity of $\eta$ and $\f F$ in the interior of 
the domain of $\eta$.

\begin{hypothesis}\label{hypo:smooth}
Let the assumptions of Hypothesis~\ref{hypo} be fulfilled.
Set $\setdom := \interi \dom \eta $ and assume that $ \eta\big|_{\setdom} \in \C^2(\setdom; \R)$ such that $ D^2 \eta (\f z ) $ is positive definite for all $\f z \in M$, 
and that $\f F\big|_{\setdom} \in \C^1(\setdom  ; \R^{m \times d} )$ such that there exists a $\tilde{\f q} \in \C^1(M;\R^d)$ fulfilling~\eqref{eq:entropyflux.new}.
\end{hypothesis}

Under these regularity assumptions, 
we can introduce the relative total entropy functional $\mathcal{R}: \domE\times 
\C^1(\T; M)\ra \R$, which is given by 
\begin{subequations}\label{RW}
\begin{equation}
\mathcal{R}(\f U| \tU) := \E(\f U) - \E(\tU) -\langle D\E(\tU),\f U- \tU\rangle \,.\label{R}
\end{equation}
Additionally, we define the relative form 
$\mathcal{W}:\domE\times 
\C^1(\T; M)\ra \R $ via
\begin{equation}
\mathcal{W}(\f U | \tU)  \!=\! \int_{\T}\! \nabla D \eta (\tU) :\! \left ( \f F( \f U) - \f F(\tU) - D \f F(\tU) ( \f U {-} \tU ) \right ) \!\de \f x + \mathcal{K}(\tU ) \mathcal R(\f U | \tU)  \,. \label{W}
\end{equation}
\end{subequations}
We note that the assumption $\tU \in  \C^1(\T ; M)$  implies 
$\tU\in\domE$,
so that $\mathcal R(\f U|\tU)$ is finite.
Indeed, since $\eta$ is continuous in the interior of its domain,
the composition $\eta\circ\tU$ is a continuous function on the compact set $\T$
and thus bounded, 
which yields $\E(\tU)<\infty$.
Similarly, all compositions of functions in \eqref{W} are bounded,
and $\mathcal W$ is well defined.
Moreover, 
both terms $\mathcal R$ and $\mathcal W$ are non-negative
due to the convexity of $\eta$ and of the function from \eqref{ass:convex},
respectively.

\begin{proposition}[Relative entropy inequality]\label{prop:weakstrong}
Let $(\f U, E)$ be an energy-variational solution in the sense of
Definition~\ref{def:envar}, and let Hypothesis~\ref{hypo:smooth} be satisfied.
 Then the relative  entropy inequality 
 \begin{equation}
 \begin{aligned}
&\left [\mathcal{R}(\f U| \tU ) + E - \mathcal E(\f U)\right ]\Big|_s^t- \int_s^t\mathcal{K}(\tU) \left [\mathcal{R}(\f U| \tU ) + E - \mathcal E(\f U)\right ] \de \tau 
\\ 
&\qquad
+ \int_s^t \left [\mathcal{W}(\f U | \tU) +\int_{\T}\left ( \t \tU + \di \f F(\tU) \right ) \cdot D^ 2\eta(\tU)  ( \f U - \tU ) \de \f x
 \right ]\de \tau\leq 0 \,\label{inuniqu}
\end{aligned}
\end{equation}
 holds for a.e.~$s,t\in(0,T)$ and all $\tU \in  \C^1(\T \times [0,T]; M) $.
\end{proposition}

An immediate consequence of inequality \eqref{inuniqu}
is the following weak-strong uniqueness property.

\begin{corollary}[Weak-strong uniqueness]\label{cor:uni}
Let Hypothesis~\ref{hypo:smooth} be satisfied.
If there exists a strong solution 
$\tU \in \C^1(s,t;\mathbb{Y})\cap\C([s,t);\mathbb{Y})$ 
to \eqref{eq.pde} in some interval $(s,t)\subset[0,T]$,
then $(\tU,\mathcal E(\tU))$ coincides with any energy-variational  solution $( \f U ,E) \in \C_{w^*}(0,T;\Y^*) \times \BV $
in the sense of Definition~\ref{def:envar} 
with $(\f U(s), E(s-))= (\tU(s), \E(\tU(s)))$. 
\end{corollary}

\begin{proof}
Since~$\tU$ is a strong solution on $[s,t]$, 
it holds $ \t \tU + \di \f F(\tU)=0$ in $(s,t)$. 
For any energy-variational solution $(\f U, E)$ such that $\f U(s) = \tU(s) $ and $E(s-) = \E(\tU(s))$, we further observe
\[
\mathcal{R}(\f U(s)| \tU(s))+ E(s-)- \E(\f U(s)) = 0 \,.
\]
From the inequality~\eqref{inuniqu}, we thus infer that
\begin{align*}
\mathcal{R}(\f U(r)| \tU(r) ) + E (r+)- \E(\f U(r))
&+ \int_s^r \mathcal{W}(\f U|\tU) \de \tau \\
&\leq \int_s^r\mathcal{K}(\tU) \left [\mathcal{R}(\f U| \tU ) + E - \E(\f U)\right ] \de \tau 
\end{align*}
for all $r \in [s,t]$.  
The convexity of the function from~\eqref{ass:convex} implies $\mathcal W\geq 0$.
From Gronwall's inequality, we infer 
that $\mathcal{R}(\f U| \tU) + E - \E(\f U)\leq 0$ in $(s,t)$. 
Since $E\geq \E(\f U)$, this implies $\mathcal{R}(\f U| \tU)\leq0$, so that
$\f U= \tU$
due to the strict convexity of $\eta$. 
\end{proof}

\begin{remark}
The above weak-strong uniqueness result is stronger than the usual weak-strong uniqueness results (cf.~\cite{weakstrongeuler}). Usually, these results are stated in the sense that: If there exists a strong solution emanating from the same initial data as the generalized solution, then both solutions coincide as long as the strong one exists. 
The above result also holds in case that the energy-variational solution coincides with a strong solution at some later point $s$ in the evolution. 
However, the solution has to satisfy $E(s-)=\E(\f U(s))$ at such a point in time.

Note that here we do not claim existence of 
such regular solutions.
There are many different results on the existence of classical solutions on short time intervals for conservation laws. We refer to~\cite[Ch.~V]{dafermos2} and the references therein. 
\end{remark}

It remains to show the relative entropy inequality \eqref{inuniqu}.

\begin{proof}[Proof of Proposition~\ref{prop:weakstrong}]
For any smooth function $\tU \in \C^1(\T\times [0,T]; M) $, we observe by the  fundamental theorem of 
calculus and the product rule that
\begin{equation}
\left [ \E(\tU) - \langle D \E(\tU),  \tU\rangle \right ]\Big|_s^t + \int_s^t \int_{\T} D^2\eta (\tU) (\f U-\tU ) \cdot  \t \tU -  \t D\eta(\tU) \cdot \f U \de \f x  \de \tau = 0 \,.\label{weakstrong1}
\end{equation}
Note that $\tU$ only takes values in $M$ such that the following calculations are rigorous. 
Taking the derivative of the assumed 
relation~\eqref{eq:entropyflux.new} with respect to $\f z$, we infer
\begin{align*}
D_l \f F_{ij}( D \eta^*(\f z)) D^2_{lk}\eta^*(\f z) =\frac{\partial}{\partial \f z_k} \f F_{ij}(D \eta^*(\f z)) = D^2_{ki} [ \f q_j \circ D\eta^* ] (\f z )   \,. 
\end{align*}
Note that since $ \f z = D \eta (D \eta^*( \f z))$, we infer from the implicit function theorem that $\eta^*$ is twice continuously differentiable with $ D^2 \eta^*(\f z) = \left [ D^2 \eta (D \eta^*( \f z))\right ]^{-1}$. 
We may express the derivative of $\f F$ via
\[
D_l \f F_{ij}( D \eta^*(\f z))  = D^2_{ki} [ \f q_j \circ D\eta^* ] (\f z )  D^2_{kl}\eta(D\eta^*(\f z))\,. 
\]
 Multiplying the above relation by $ D^2\eta(D \eta^*( \f z))$ from the left, we infer by the symmetry of the second derivatives of $\f q$ and $\eta$ that
\begin{align*}
D^2_{im } \eta(D \eta^*( \f z)) D_l \f F_{ij}( D \eta^*(\f z)) &= D^2_{im } \eta(D \eta^*( \f z))D^2_{ki} [ \f q_j \circ D\eta^* ] (\f z )  D^2_{lk } \eta(D \eta^*( \f z))\\& =D^2_{lk } \eta(D \eta^*( \f z)) D_m \f F_{kj}( D \eta^*(\f z))\,.
\end{align*}
This symmetry can be used to calculate 
\begin{align*}
D^2_{im } \eta( \tU )\frac{\partial }{\partial{\f x_j}} \f F_{ij} ( \tU)  &= D^2_{im } \eta( \tU ) D_l \f F_{ij}(\tU) \frac{\partial \f U_l}{\partial \f x_j}\\& = D^2_{il } \eta( \tU ) D_m \f F_{ij}(\tU) \frac{\partial \f U_l}{\partial \f x_j} = \frac{\partial}{\partial \f x_j } D_i\eta(\tU)  D_m \f F_{ij}(\tU) \,,
\end{align*} 
which implies
\[
\begin{aligned}
 \di \f F (\tU) \cdot D ^2\eta(\tU ) ( \f U-\tU) &=  (D \f F (\tU): \nabla \tU) \cdot   D ^2\eta(\tU) ( \f U-\tU) \\
&=  D \f F (\tU)\dreidots \nabla   D \eta(\tU)\otimes  ( \f U-\tU) \\& =\nabla   D \eta(\tU): \left ( D \f F (\tU) ( \f U-\tU)\right )\,.
\end{aligned}
\]
Additionally, we may set $\Phi :=D \eta (\tU) $ in~\eqref{eq:integralFentropy} in order to conclude from 
the Fenchel equivalences~\eqref{eq:fenchel}  that
\[
\int_{\T} \nabla D \eta (\tU):  \f F(\tU)\de \f x =\int_{\T} \nabla \Phi :  \f F(D \eta^* ( \Phi))\de \f x =0\,.
\]
Combining the last two equations, we find
\begin{align}
0 = \int_{\T} \di \f F (\tU) \cdot D ^2\eta(\tU ) ( \f U-\tU)  \de \f x  - \int_{\T}  \nabla D \eta (\tU): \left [ \f F(\tU)+ D \f F (\tU) ( \f U-\tU)\right ] \de \f x \,.\label{weakstrong2}
\end{align}
Adding the above identities~\eqref{weakstrong1} and~\eqref{weakstrong2} to the inequality~\eqref{envarform} with $\Phi = D \eta(\tU)$ implies
\begin{align*}
&\left [ E - \E(\tU) - \langle D \E (\tU) , \f U - \tU \rangle \right ]\Big|_s^t \\
&\qquad+ \int_s^t \int_{\T} \nabla D \eta(\tU ): \left ( \f F(\f U) -\f F (\tU) - D \f F (\tU) ( \f U-\tU)\right )  \de \f x \de \tau  \\
&\qquad+\int_s^t\int_{\T}\left ( \t \tU + \di \f F(\tU) \right ) \cdot D^ 2\eta(\tU) \left ( \f U - \tU\right )  \de \f x+ \mathcal{K}(\tU) \left [\E (\f U) - E \right ]  \de \tau \leq 0 
\end{align*}
which is \eqref{inuniqu}.
\end{proof}

\subsection{Existence of energy-variational solutions}
\label{subsec:existence}

In this subsection 
we prove Theorem \ref{thm:main}, that is, we show existence
of energy-variational solutions to the hyperbolic conservation law \eqref{eq}.
To do so, we introduce a semi-discretization scheme in time.
For $N\in\N$, we define $\tau :=T/N$, and we set $t^n:= \tau n$ for $n\in \{ 0,\ldots , N\}$
to obtain 
an equidistant partition of $[0,T]$.
We set $\f U^0\coloneqq\f U_0\in\domE$,
and in the $n$-th time step, $n\geq 1$,
we compute $\f U^n$ from $ \f U^{n-1}\in \domE$
by solving the minimization problem
\begin{equation}\label{eq:timedis}
\begin{aligned}
  \f U^n= \argmin _{\f U\in \domE; \E(\f U)\leq \E(\f U^{n-1}) }  \, &\sup _{ \Phi \in \Y
  } \!
\Bigg [
\left (\E(\f U) - \E(\f U^{n-1})\right ) - \left ( \f U - \f U^{n-1}, \Phi \right ) 
\\
&\quad
+ \tau \left [ \F{\f U}{\Phi} + \mathcal{K}(\Phi ) \left( \E(\f U ) - \E(\f U^{n-1})  \right ) \right ] \Bigg] \,.
\end{aligned}
\end{equation}
\begin{remark}[Comparison to time discretization for gradient flows]
In the theory of gradient flows it is nowadays standard to consider a time-discretization scheme based on a sequential minimization~\cite[Chap.~6]{gradient}. 
This is certainly a different setting than in the problem considered here
since the energy is not formally conserved along a gradient flow
but dissipated by some dissipation functional. Nevertheless, 
a similarity is that a saddle-point problem has to be solved in every time step. The current algorithm can thus be seen as a first generalization of this technique from gradient flows to more general systems, also including Hamiltonian dynamics. A goal for the future is to combine both approaches in order to find a suitable discretization scheme for general GENERIC systems~\cite{generic}, which combine dissipative and Hamiltonian effects. 
\end{remark}

\begin{remark}[Solving the min-max problem numerically]
It is worth observing that the discrete optimization problem
from~\eqref{eq:timedis} is given in form of a saddle-point problem. 
This is a standard problem in optimization theory and machine learning and there are different tools to solve such a problem numerically~\cite{numeric}.
\end{remark}

\begin{theorem}[Solution of the time-discrete problem]\label{thm:disex}
For each $\f U^{n-1} \in \domE$ 
there exists a unique solution $ \f U^n$ to the minimization problem~\eqref{eq:timedis},
and it holds 
\begin{align}
 \left (1 +\tau \mathcal{K}(\Phi)\right )  \left ( \mathcal{E}(\f U^n) - \mathcal{E}(\f U^{n-1})\right ) - \langle \f U^n - \f U^{n-1}, \Phi\rangle + \tau \F{\f U^n}{\Phi}
 \leq 0 \label{disineq}
\end{align}
for all $\Phi \in\Y$.
\end{theorem}
\begin{proof}
The proof is divided into different steps:

\textit{Step 1: Functional framework.}
We define the set
\[
\mathbb{D} ^n :={} \left \{ \f U \in \domE | \, \E(\f U ) \leq  {\E(\f U^{n-1})}{}\right \} 
\]
and the function
\begin{align*}
\mathcal{F}_n^\tau(\f U | \Phi) :={}& \left ( 1 + \tau \mathcal{K}(\Phi) \right ) \left ( \E (\f U ) - \E(\f U^{n-1}) \right ) 
- \left \langle \f U - \f U^{n-1} , \Phi \right \rangle 
+ \tau \F{\f U}{\Phi}  \,.
\end{align*}
Then we solve the time-discrete minimization problem \eqref{eq:timedis}
if we find a unique minimizer $\f U^n\in\domE^n$ of the function
\[
\mathcal H \colon \mathbb D^n\to\R,
\quad
\mathcal H(\f U)
=\sup_{\Phi  \in \mathbb{Y} } \mathcal{F}_n^\tau(\f U | \Phi  )\,.
\]

\textit{Step 2: Min-max theorem.}
In order to show that 
\begin{equation}
\inf_{\f U \in \mathbb{D}^n
} \sup_{\Phi \in\mathbb{Y}} \mathcal{F}_n^\tau(\f U| \Phi  ) = \sup _{\Phi\in\mathbb{Y}} \inf_{\f U \in \mathbb{D}^n
}\mathcal{F}_n^\tau(\f U| \Phi ) \,.\label{minmax}
\end{equation}
we apply a min-max theorem.
Since $\E$ is superlinear,
the set $\domE^n $ is weakly compact in $L^1(\T;\R^m)$ by Lemma~\ref{lem:delavalle}
and the function 
$ \f U \mapsto \mathcal{F}^\tau_n  ( \f U | \Phi)$ is convex and weakly lower semi-continuous for every $ \Phi \in\Y $. 
Moreover, the function $ \Phi \mapsto \mathcal{F}^\tau_n(\f U| \Phi)$ is concave for all $ \f U \in \mathbb D^n$
since $\mathcal K$ is convex and $\E(\f U)\leq\E(\f U^{n-1})$.
Therefore, \eqref{minmax} follows from Fan's 
min-max theorem~\cite[Theorem 2]{Fan1953}.

\textit{Step 3: Inequality~\eqref{disineq}.}
We show $ \inf_{\f U \in \domE^n}\mathcal H(\f U) \leq 0$.
To do so, 
let $\Phi \in \mathbb{Y} $ be arbitrary and define $ \tU =D\eta^\ast\circ\Phi$
and $\hat{\f U} =  D \eta ^* \circ( \alpha \Phi ) $,
where $\alpha>0$ is chosen as follows:
If $\E(\tU) \leq \E(\f U^{n-1})$, we set $\alpha =1$,
so that $\hat{\f U} = \tU$. 
If $\E(\tU) > \E(\f U^{n-1})$, we let $\alpha\in(0,1)$ 
such that $\E ( \hat{\f U}) = \E(\f U^{n-1})$,
which is possible by Lemma \ref{lem:surjective}. 
Then
the assumed identity \eqref{eq:integralFentropy}
implies
\[
\int_\T \f F(\hat{\f U }
) : 
\nabla \Phi
 \de \f x 
=\frac{1}{\alpha}\int_\T \f F( D\eta^\ast(\alpha\Phi(x))):\alpha\nabla\Phi(x)\,\dx
=0.
\]
Since
$\alpha\Phi \in\partial \E(\hat{\f U})$,
from the definition of the subdifferential
 of $\E$ we obtain 
\[
\begin{aligned}
 \inf_{\f U\in\domE^n} &\mathcal{F}^\tau_n ( \f U | \Phi)
 \leq \mathcal{F}^\tau_n ( \hat{\f U } |  \Phi )
 \\
 &= \left ( 1 + \tau \mathcal{K}(\Phi) \right ) \left ( \mathcal{E}(\hat{\f U }) - \E(\f U^{n-1})\right )  
 -\frac{1}{\alpha}  \left \langle\hat{\f U } - \f U^{n-1} ,\alpha\Phi 
 \right \rangle + \tau \int_{\T} \f F(\hat{\f U }) : \nabla \Phi \de \f x 
\\ 
&\leq  \left (1 + \tau \mathcal{K}(\Phi) -\frac{1}{\alpha}   \right ) \left ( \mathcal{E}(\hat{\f U }) - \E(\f U^{n-1})\right ) \leq  0\,.
 \end{aligned}
\]
The last inequality 
follows since $\alpha\in(0,1]$ and 
$\E(\hat{\f U}) \leq \E(\f U^{n-1})$.
Because $\Phi\in\Y$ was arbitrary,
identity \eqref{minmax} implies $\inf_{\f U \in \domE^n}\mathcal H(\f U)\leq 0$.

\textit{Step 4: Solvability of the optimization problem.}
From the identity
\[
\begin{aligned}
\mathcal H(\f U)
&= \left ( \E (\f U ) -  \E(\f U^{n-1}) \right )  \\
&\quad +     \sup_{\Phi\in \mathbb{Y}} 
  \left ( \tau \mathcal{K}(\Phi ) (\E(\f U)-  \E(\f U^{n-1}) )- 
\left \langle \f U - \f U^{n-1} , \Phi  \right \rangle 
+\tau \F{\f U}{\Phi }  \right ),
\end{aligned}
\]
we conclude the strict convexity 
of the mapping
$\mathcal H $
from the strict convexity of $\E$ and the convexity the function in the second line,
which is the supremum of convex functions. 
Additionally, $\mathcal H $
is not equal to $+\infty $ everywhere due to \textit{Step 3}. 
Furthermore, we observe the coercivity of $\mathcal H$ 
via 
\[
\mathcal H(\f U) 
\geq \mathcal{F}^\tau_n(\f U | \f 0)  =  \E (\f U ) -  \E(\f U^{n-1})\,
\]
since $\E$ is superlinear,
which also implies that 
$\mathbb D^n $ is weakly compact in $L^1(\T)$ by Lemma~\ref{lem:delavalle}. 
In total, $\mathcal H$ is a 
strictly convex, lower semicontinuous and coercive function on the compact set $\domE^n$ 
and thus has a unique minimizer $\f U^n$. 
\end{proof}

\begin{proof}[Proof of Theorem~\ref{thm:main}]
We prove the existence of  energy-variational solutions
via the convergence of a time-discretization scheme. 
We divide the proof into three steps.

\textit{Step 1: Discretized formulation.}
For $N\in\N$, let $\tau = T/N$ and $t^n= n \tau $ as above.
Set $\f U^0=\f U_0$,
define $\f U^n$ iteratively
by~\eqref{eq:timedis},
and set $E^n \coloneqq \E(\f U^n)$ for  $n\in\set{0,\dots,N}$.
Theorem~\ref{thm:disex} guarantees that $\f U^n\in\domE$ exists and satisfies
\begin{equation}\label{disrelen}
\begin{aligned}
E^n - E^{n-1}  & - \left \langle \f U^n - \f U^{n-1} , \Phi  \right \rangle 
\\
&+ \tau \left [\F{\f U^n}{\Phi }  +\mathcal{K}(\Phi ) [ \E(\f U^n)   -  E^{n-1}] 
\right ] \leq 0 \,
\end{aligned}
\end{equation}
for all $\Phi \in \mathbb{Y} $. 
For functions $\phi \in \C^\infty_c([0,T); [0,\infty))$ and $ \Phi \in \C^1( [0,T];\mathbb Y )$,
we define $ \phi^n \coloneqq \phi(t^n)$ and $ \Phi^n \coloneqq \Phi (t^n)$  for $n \in \{ 0, \ldots , N\}$.
Using $\Phi = \Phi ^{n-1}$ in~\eqref{disrelen},
multiplying the resulting inequality by $\phi^{n-1}$ and summing this relation
over $n\in \{ 1, \ldots , N\}$ implies 
\[
\begin{aligned}
\sum_{n=1}^N &\left [ \phi^{n-1} ( E^n-E^{n-1} ) - \phi^{n-1} \langle \f U^n-\f U^{n-1} , \Phi^{n-1} \rangle \right ]
\\
&+ \tau \sum_{n=1}^N \phi^{n-1} \left [  \F{\f U^n}{\Phi ^{n-1}} + \mathcal{K}(\Phi ^{n-1} ) (\E(\f U^n) -E^{n-1}) \right ] \leq 0 \,.
\end{aligned}
\]
Since $\phi^N=0$,
using a discrete integration-by-parts formula 
and dividing by $\tau>0$, 
we obtain
\begin{equation}
\begin{aligned}
-\sum_{n=1}^N &\left [\frac{\phi^n- \phi^{n-1}}{\tau} ( E^n- \left \langle  \f U^n,\Phi^{n-1}\right \rangle)  - \phi^{n} \left \langle \f U^n ,\frac{ \Phi^n- \Phi^{n-1}}{\tau}\right  \rangle \right ]- \phi^0 \bp{ \E(\f U_0)-\langle \Phi^0, \f U_0\rangle }
\\
&+ \sum_{n=1}^N \phi^{n-1} \left [  \F{\f U^n}{\Phi ^{n-1}} + \mathcal{K}(\Phi ^{n-1} ) (\E(\f U^n) -E^{n-1}) \right ] \leq 0 \,.
\end{aligned}
\label{disrel}
\end{equation}

\textit{Step 2: Prolongations.}
We define the piece-wise constant prolongations
\[
\begin{aligned}
\ov{\f U}^N(t) &:= \begin{cases}
\f U^n & \text{for } t \in (t^{n-1},t^n], \\
\f U_0 & \text{for } t = 0,
\end{cases} \,\quad
\\
\ov {E}^N(t) &:= \begin{cases}
\mathcal{E}(\f U^{n})
& \text{for } t \in (t^{n-1},t^n],
 \\
\mathcal{E}(\f U_0) &\text{for } t = 0,
\end{cases}\,
\qquad
\un{E}^N(t) := \begin{cases}\mathcal{E}(\f U^N)  & \text{for } t  = T, \\
\mathcal{E}(\f U^{n-1} ) & \text{for } t \in [t^{n-1},t^n).
\end{cases}\,
\end{aligned}
\]
Analogously, for  test functions $ \psi \in \C^1([0,T]; \mathbb{X})$,
where $\mathbb{X}$ is $\R$ or $\mathbb{Y}$,
we define the piece-wise constant and piece-wise linear prolongations by
\begin{align*}
    \overline{\psi}^N(t)& \coloneqq \begin{cases}
\psi(t^n) & \text{for } t \in (t^{n-1},t^n], \\
\psi(0) & \text{for } t = 0,
\end{cases} \,\qquad
      \underline{\psi}^N(t) \coloneqq \begin{cases} \psi(T)  & \text{for } t  = T, \\
\psi(t^{n-1}) & \text{for } t \in [t^{n-1},t^n),
\end{cases}\\
     \widehat{\psi}^N(t)& \coloneqq \frac{\psi(t^n)-\psi(t^{n-1})}{\tau}(t-t^{n-1}) + \psi(t^{n-1})
      \quad \text{for } t\in[t^{n-1},t^n]\,.
\end{align*}
With this notation, the discrete energy-variational inequality \eqref{disrel}
becomes
\begin{multline}\label{disenin}
- \int_0^T \Bp{ \t \widehat{\phi} ^N\left [ \ov E ^N - \langle \ov {\f U}^N, \un{\Phi}^N  \rangle \right ] - \ov \phi^N \langle \ov{\f U}^N  , \t \widehat{\Phi}^N\rangle  + \un{\phi}^N \mathcal{K}(\un{\Phi}^N )\un{E}^N} \de t 
\\
+ \int_0^T \un\phi^N \left [\F{\ov{\f U}^N}{\un{\Phi}^N}+ \mathcal{K}(\un{\Phi}^N )  \E(\ov{\f U}^N) \right ]\de t - \phi(0)\left[\E(\f U_0) - \langle \Phi(0) , \f U_0 \rangle \right] \leq 0 
\end{multline}
for all $ \Phi \in \C^1( [0,T];\mathbb Y )$ 
and all $\phi \in \C^1_c([0,T))$ with $\phi \geq 0$. 

\textit{Step 3: Convergence.}
Since we have 
$0\leq\E(\f U^{n})\leq \E(\f U^{n-1})$, 
we obtain that $t\mapsto \ov{E}^N(t)  $ and  $t\mapsto \un{E}^N(t)  $
are non-negative and non-increasing functions and as such bounded in $\BV$ by the initial value $E^0=\E(\f U_0)$. 
Moreover, by the superlinear growth of $\eta $, we infer 
from $ \E ( \ov{\f U}^N(t)) \leq \E(\f U_0)$
 that the sequence $\{ \ov{\f U}^N\}_{N\in\N}$ is bounded in $L^\infty(0,T;\domE)$.
Thus, we may extract (not-relabeled) subsequences 
 such that there exist $\ov{E},\un{E}\in \BV$ and $\f U \in L^\infty(0,T;\domE)$ such that
\[
\begin{aligned}
\ov{\f U}^n &\xrightharpoonup{(\ast)} \f U 
&&\quad\text{in } L^\infty(0,T;L^1(\T;\R^m))\,,\\
( \ov{E}^N,\,\un{E}^N) &\xrightharpoonup{\ \ast\ } (\ov{E},\, \un{E}) 
&&\quad \text{in }\BV\,,\\
( \ov{E}^N(t),\,\un{E}^N(t)) &\xrightarrow{\ \phantom{\ast}\ } (\ov{E}(t),\, \un{E}(t)) 
&&\quad \text{for all }t\in[0,T]\,,
\end{aligned}
\]
where the weak$(^\ast)$ convergence in 
$L^\infty(0,T;L^1(\T;\R^m))$ was defined in \eqref{eq:weakconv.LinfL1},
and where we used Helly's selection theorem (see~\cite[Thm.~1.126]{barbu} for example).
We next show that $\ov{E}^N $ and $\un{E}^N$ converge to the same limit,
that is, $\ov{E}=\un{E}$ a.e.~in $(0,T)$. 
Due to the monotony~$E^n \leq E^{n-1}$, we find
\[
\int_0^T | \ov{E}^N-\un{E}^N|\de t 
= \sum_{n=1}^N \tau   (E^{n-1}-E^{n} ) 
= \tau( E(0) - E^N) \leq\tau  E(0) \longrightarrow  0 \quad 
\text{ as } N \to\infty\,.
\]
Since $\BV$ continuously embeds into $L^1(0,T)$,
this allows to identify $\un{E}=\ov{E}=:E$.  
Due to the pointwise convergence~in $[0,T]$ of $\ov{E}^N$, 
we infer from
the weak lower semi-continuity of $\E$  that $ E \geq \E(\f U)$ a.e. in $(0,T)$. 
We 
clearly have
\begin{align*}
\t \widehat{\phi}^N &\ra \t \phi, 
& \ov{\phi}^N &\ra \phi, 
& \un{\phi}^N &\ra \phi  
&&\text{ pointwise in } [0,T] \text{ as }N \ra \infty \,,
\\
\t \widehat{\Phi}^N &\ra \t \Phi, 
& \un{\Phi}^N &\ra \Phi, 
& \nabla\un{\Phi}^N &\ra \nabla\Phi \quad &\text{in } \C(\T;\R^m) 
&\text{ pointwise in } [0,T] \text{ as }N \ra \infty 
\,.
\end{align*}
With these observations, we may pass to the limit in the weak form~\eqref{disenin}. 
We note that $\ov{\f U}^N$  occurs linearly in the first line of~\eqref{disenin}. All other terms are bounded and converge almost everywhere in $(0,T)$.
This implies that 
\begin{multline*}
\lim_{N\ra\infty} \int_0^T \Bb{ \t \widehat{\phi} ^N\left [ \ov E ^N - \langle \ov {\f U}^N , \un{\Phi}^N \rangle \right ] - \ov \phi^N \langle \ov{\f U}^N  , \t \widehat{\Phi}^N\rangle  + \un{\phi}^N \mathcal{K}(\un{\Phi}^N )\un{E}^N }\de t \\
= \int_0^T\Bb{ \t {\phi}\left [  E  - \langle   {\f U} , {\Phi}, \rangle \right ] -  \phi \langle {\f U}  , \t {\Phi}\rangle  + {\phi} \mathcal{K}({\Phi} ){E}} \de t\,.
\end{multline*}
Observing that the second line in~\eqref{disenin} is bounded from below due to Hypothesis~\eqref{hypo} and that $\phi\geq 0$ in $[0,T]$, we may apply Fatou's lemma and the weak lower semi-continuity of the function from~\eqref{ass:convex} as well as the continuity of $\mathcal{K}$ in order to pass to the limit in the second line of~\eqref{disenin}, which yields
\begin{multline*}
\liminf_{N\ra\infty} \left [ \int_0^T \un\phi^N \left [\F{\ov{\f U}^N}{\un{\Phi}^N}+ \mathcal{K}(\un{\Phi}^N )  \E(\ov{\f U}^N) \right ]\de t \right ] \\
\geq  \int_0^T\liminf_{N\ra\infty} \Bb{ \un\phi^N \left [\F{\ov{\f U}^N}{\un{\Phi}^N}+ \mathcal{K}(\un{\Phi}^N )  \E(\ov{\f U}^N) \right ]}\de t \\
 \geq \int_0^T \phi \int_{\T} \f F (\f U) : \nabla \Phi \de \f x + \mathcal{K}(\Phi  ) \mathcal{E}(\f U)  \de t\,.
\end{multline*}
In total, we infer from \eqref{disenin} that
\[
\begin{aligned}
-\int_0^T &\t  \phi \left [ {E} - \left \langle {\f U } ,  \Phi \right \rangle\right ]  \de t - \phi(0)\left[\E(\f U_0) - \langle  \f U_0,\Phi(0)  \rangle \right]\\
&+\int_0^T \phi  \left [\left \langle {\f U } , \t  \Phi \right \rangle+ \F{\f U}{\Phi } +    \mathcal{K}(\Phi ) [\E({\f U}) - E]   \right ] \de t 
 \leq 0\,.
\end{aligned}
\]
Via Lemma~\ref{lem:invar}, we now end up with the energy-variational inequality~\eqref{envarform} and with
\begin{equation*}
\lim_{t \searrow 0}\left [ E (t) - \langle \f U(t),\Phi(t)  \rangle \right] \leq \E(\f U_0) - \langle\f U_0, \Phi(0)  \rangle\,,
\end{equation*}
after possible redefining the function on a set of measure zero.
By Lemma~\ref{lem:var.affine},
this inequality implies $\f U(t+)=\f U_0$ in $\mathbb Y^\ast$, that is, 
the initial value is attained in the asserted sense.
\end{proof}

\section{Two incompressible fluid models\label{sec:incomp}}
Our first two examples are models
for incompressible inviscid fluids, the incompressible magnetohydrodynamical equations and the incompressible Euler system. 
While the latter can be seen as a special case of the first system,
it allows to derive more properties 
and the comparison
with weak dissipative solutions for the Euler equations.

\subsection{Incompressible magnetohydrodynamics}
\label{sec:magneto}
As the first example, we consider the equations modeling 
an incompressible, inviscid and electronically conductive fluid.
The corresponding equations of motion are the 
magnetohydrodynamical equations given by
\begin{subequations}\label{eq:magneto}
\begin{align}
\t \f v + ( \f v \cdot \nabla ) \f v -  \mu (\f H \cdot \nabla ) \f H  + \nabla p + \nabla \frac{\mu}{2}| \f H|^2 ={}& \f 0 ,  \qquad && \text{in }\T \times (0,T)\,,\label{eq:incompNav}
\\
\t \f H - \curl  ( \f v \times \f  H )  ={}& \f 0 \qquad && \text{in }\T \times (0,T)\,,\label{eq:mag}
\\
 \di \f v = 0 , \quad \di \f H ={}& 0 \qquad && \text{in }\T \times (0,T)\,,
 \label{eq:magneto.div}
\\
\f v (0) = \f v_0, \quad \f H(0) = {}&\f H_0 \qquad && \text{in } \T \,.
\end{align}
\end{subequations}
Here $ \f v : \T\times (0,T) \ra \R^d $ denotes the velocity of the fluid, $ \f H : \T \times (0,T) \ra \R^d $ is the magnetic field,  $p : \T\times (0,T) \ra \R$ denotes the pressure, and $\mu \in (0,\infty)$ is the quotient of the magnetic permeability and the constant density of the fluid. 

\begin{remark}
We note that the above equation is not formally of the form~\eqref{eq.pde}. The pressure is not a function of $\f H $ and $\f v$ but should rather be seen as a Lagrange multiplier to fulfill the divergence-free condition in the evolution. 
The first equation~\eqref{eq:incompNav} can be interpreted as $\t \f v+ P\di( \f v\otimes \f v - \mu \f H \otimes \f H 
) =0$, where $P$ denotes the Helmholtz projection on divergence-free functions, 
and condition \eqref{eq:magneto.div} is incorporated in the functional framework
by working in the space of divergence-free functions.
Another viewpoint is that one can 
derive a weak formulation of \eqref{eq:magneto}
by testing with divergence-free test functions.
Then the pressure term can be omitted,
and the weak formulation is of the form~\eqref{eq:weak.intro}.
\end{remark}

To introduce the notion of energy-varational solutions 
to the magnetohydrodynamical equations \eqref{eq:magneto},
we define the corresponding mathematical entropy
as the physical energy
\begin{equation}\label{eq:magneto.energy}
\E (\f v, \f H)
= \frac{1}{2}\norml{\f v }_{L^2(\T)}^2+ \frac{\mu}{2}\norml{\f H}_{L^2(\T)}^2.
\end{equation}
Moreover, we 
we introduce the class of divergence-free vector fields
\[
\LRsigma{q}(\T)
\coloneqq\setcL{\f v\in\LR{q}(\T;\R^d)}{\forall \f \varphi\in \C^1(\T)\colon \int_\T \f v \cdot \nabla\f \varphi\,\dx =0}
\]
for $q\in[1,\infty)$. 
The mathematical precise sense of energy-variational solutions
is given in the following definition.

\begin{definition}\label{def:magneto}
A tuple $(\f v , \f H , E )\in  L^\infty(0,T; L^2_\sigma ( \T ) )^2  \times \BV $  is called an energy-variational solution to the incompressible  magnetohydrodynamical equations
\eqref{eq:magneto} 
if it satisfies $\E (\f v(t), \f H(t))\leq E(t)$ for a.a.~$t\in (0,T)$,
and if the energy-variational inequality
\begin{equation}
\begin{aligned}
&\left [ E - \int_{\T} 
 \f v  \cdot \f \varphi -  \f H \cdot \f \psi 
 \de \f x  \right ] \Big|_s^t 
+ \int_s^t \int_{\T}\bb{
\f v \cdot \t \f \varphi 
+ \left  (\f v  \otimes \f v - \mu \f H \otimes \f H \right  ) : \nabla \f \varphi 
\de \f x 
\\
&\quad
+ \int_s^t \int_{\T}
\f H \cdot \t \f \psi  + \left  (\f H  \otimes \f v -  \f v \otimes \f H   \right ) :\nabla \f \psi 
\de \f x 
+ \mathcal{K}(\f \varphi, \f \psi  ) \left [ \E(\f v, \f H ) - E\right ] 
}\de  s
\leq 0
\,
\end{aligned}
\label{relenMagneto}
\end{equation}
holds for a.e.~$s<t\in(0,T)$ including $s=0$ with $( \f v (0), \f H (0)) = ( \f v_0 ,\f H_0)$ and all test functions $(\f \varphi, \f \psi )  \in \C^1(\T \times [0,T]; \R^{d})^2$ with $ \di \f \varphi =  \di \f \psi = 0$.
Here, 
\begin{equation}
\mathcal{K}(\f \varphi, \f \psi ) =
2\|(\nabla \f \varphi )_{\sym}\| _{L^\infty(\O;\R^{d\times d})} + \frac{2}{\sqrt\mu} \|(\nabla \f \psi )_{\skw}\| _{L^\infty(\O;\R^{d\times d})}\label{K:magneto}
\end{equation}
with
\[
\| \Phi\| _{L^\infty(\O;\R^{d\times d})} 
=
\esssup_{x\in\O} | \Phi(x) |_2,
\]
where $\snorm{\cdot}_2$ denotes the spectral norm defined in \eqref{eq:spectralnorm}.
\end{definition}
 \begin{theorem}\label{thm:magneto}
For every initial datum $(\f v _0, \f H_0)\in L^2_{\sigma}(\T)\times L^2_{\sigma}(\T)$, there exists an energy-variational solution in the sense of Definition~\ref{def:magneto} with $E(0)=\mathcal{E}(\f v _0, \f H_0)$ and $ (\f v , \f H) \in \C_w([0,T];L^2_{\sigma}(\T)\times L^2_{\sigma}(\T))$, and the initial values are attained in the strong sense. 
\end{theorem}
 \begin{proof}
 We have to show that the Hypothesis~\ref{hypo} is fulfilled.
 To realize the system~\eqref{eq:magneto}
  in the abstract framework introduced above,
we introduce the quadratic entropy functional $ \eta : \R^{2d} \ra \R$ via 
$\eta(\f y _1 , \f y_2 )=\frac{1}{2}\snorm{\f y_1}^2+\frac{\mu}{2}\snorm{\f y_2}^2$,
which is obviously strictly convex, lower semi-continuous and 
has superlinear growth.
The space of test functions is given by
$\Y=\setcl{(\f\varphi,\f\psi)\in\C^1(\T;\R^d)^2}{\di\f\varphi=\di\f\psi=0}$,
and we have
$\domE=
\LRsigma{2}(\T)\times \LRsigma{2}(\T)$
(see Remark~\ref{rem:domE}),
which is obviously convex.
Note that $\eta^\ast ( \f z_1 , \f z _2) =\frac{1}{2}\snorm{\f z_1}^2+\frac{1}{2\mu }\snorm{\f z_2}^2 $,
and the corresponding total entropy 
is given by the physical energy $\E$ from \eqref{eq:magneto.energy}.
The function $\f F: \R^ {2d} \ra \R^{2d\times d } $ is given by 
$$ \f F(\f v, \f H  ) = \begin{pmatrix}
 \f v \otimes \f v - \mu \f H \otimes \f H \\  \f H \otimes \f v - \f v \otimes \f H  \end{pmatrix}
  \,.$$
Observing that $D\eta^*(\f z_1, \f z _2 ) = ( \f z_1, \frac{\f z_2}{\mu})^T$, we find  that the condition~\eqref{eq:integralFentropy} is fulfilled  due to 
\begin{align*}
\int_{\T} &\f F( D \eta^*( \f \varphi , \f \psi) ) : \nabla \begin{pmatrix}
\f \varphi\\\f \psi 
\end{pmatrix}\de \f x \\
&= \int_{\T} \left ( \f \varphi \otimes \f \varphi - \frac{1}{\mu} \f \psi \otimes \f \psi \right ) : \nabla \f\varphi + \frac{1}{\mu} \left ( \f \psi \otimes \f \varphi - \f \varphi \otimes \f \psi \right ) : \nabla\f  \psi \de \f x 
\\
&=\int_{\T} ( \f \varphi \cdot \nabla )\frac{|\f \varphi|^2}{2} - \frac{1}{\mu}\left [\left ( \f \psi \otimes \f \psi \right ) : \nabla \f \varphi 
- ( \f \varphi \cdot \nabla )\frac{|\f \psi|^2}{2} - \nabla \f\varphi : \left ( \f \psi \otimes \f \psi  \right ) - \f \varphi \cdot \f \psi \di \f\psi 
\right ]\de \f x 
\\&=0\,,
\end{align*}
where we integrated by parts in the last term. 
The last equality follows 
by another integration by parts 
since $\f \varphi$ and $\f \psi$ are solenoidal vector fields. 
Moreover, inequality~\eqref{BoundF} is fulfilled for $C= 2 + \frac{2}{\sqrt\mu}$
since from Young's inequality, it follows
\begin{align*}
| \f F(\f v , \f H )| \leq \snorm{\f v}^2 + {\mu} \snorm{\f H}^2 + 2 \snorm{\f v}\snorm{\f H} \leq 2 \eta(\f v, \f H) + \frac{2}{\sqrt \mu} \eta(\f v, \f H) \,.
\end{align*}
Finally, we have to show that the choice~\eqref{K:magneto} 
of the regularity weight $\mathcal{K}$ yields the convexity of
the function from~\eqref{ass:convex}. 
We infer similarly to the previous estimate that
\begin{align*}
\snormL{\int_{\T} \f F (\f v , \f H ) : \nabla \begin{pmatrix}
\f \varphi \\\f \psi 
\end{pmatrix} \de\f x }
 \leq {}& \left ( \| \f v \otimes \f v- \mu  \f H \otimes \f H \|_{L^1(\T;\R^{d\times d})} \right ) \| (\nabla \f \varphi )_{\sym}\|_{L^\infty(\T;\R^{d\times d})} 
\\& + 2\| \f v \otimes \f H \|_{L^1(\T;\R^{d\times d})} \| (\nabla \f \psi  )_{\skw}\|_{L^\infty(\T;\R^{d\times d})} 
\\ \leq {}&
\mathcal{K}(\f \varphi , \f \psi ) \E(\f v, \f H)\,. 
\end{align*}
This implies that the mapping 
\[ (\f v, \f H ) \mapsto \int_{\T} \f F( \f v, \f H) : \nabla \begin{pmatrix}
\f \varphi \\\f \psi 
\end{pmatrix} \de\f x +\mathcal{K}(\f \varphi , \f \psi ) \E(\f v, \f H)  \]
is quadratic and non-negative, and thus convex and weakly lower semi-continuous. 
In total, Hypothesis \ref{hypo} is satisfied,
and from Theorem~\ref{thm:main}
we infer the existence of a solution in the sense of
Definition~\ref{def:magneto}
with the regularity from Proposition~\ref{prop:reg}. 
Finally, 
Proposition~\ref{prop:betterreg} implies the additional regularity. 
\end{proof}

\begin{remark}[Alternative choice of $\f F$]
We could also define the function  $\f F$ by
$$ \f F(\f v, \f H  ) = \begin{pmatrix}
 \f v \otimes \f v - \mu \f H \otimes \f H + I \left ( \frac{| \f v |^2}{2}+ \frac{\mu|\f H|^2}{2}\right ) \\  \f H \otimes \f v - \f v \otimes \f H - ( \f v \cdot \f H) I   \end{pmatrix}
  \,.$$
  With this definition, we can derive the relation~\eqref{eq:entropyflux.new} for the function $\tilde{\f q}$ given by
 $$ 
 \tilde{\f q} (\f v , \f H) = \f v \left ( \frac{| \f v |^2}{2}+ \frac{\mu|\f H|^2}{2}\right ) - \mu \f H ( \f H \cdot \f v) \,,
 $$
and the function $\f F$ fits better into our abstract framework with Hypothesis~\ref{hypo:smooth}. But since both choices yield the same when tested with solenoidal functions, we rather use the simpler  version in the above proof. Note that both choices fulfill the condition~\eqref{eq:integralFentropy}. 
 
\end{remark}
\begin{remark}[Boundary conditions]\label{rem:boundary}
The concept can be transferred to the usual impermeability boundary conditions.
Indeed, on a bounded Lipschitz domain $\Omega \subset \R^d$, we may equip the system~\eqref{eq:magneto} with the boundary conditions $ \f n \cdot \f v = 0 = \f n \cdot \f H$ on $\partial \Omega$,
where $\f n$ denotes the outer unit normal vector at $\partial\Omega$. 
The associated space for the test functions $\mathbb Y$  has to be restricted to $ (\f\varphi, \f \psi)\in \Y := \C^1(\Omega \times [0,T]; \R^{2d}) $ with $ \f n \cdot \f \varphi = 0 = \f n \cdot \f \psi $ on $\partial \Omega$ and $\di \f \varphi = 0 = \di \f \psi $ in $\Omega$. 
Similarly to the above calculation, one may verify that condition~\eqref{eq:integralFentropy} is still fulfilled, where the integral is taken over $\Omega$ instead of $\T$. 
\end{remark}

\subsection{Incompressible Euler equations}
\label{sec:incompEuler}
For the sake of completeness, 
we apply the abstract result to the incompressible Euler equations, even though 
the existence of energy-variational solution to this system
was already proven in~\cite{envar}. 
Actually, this can be seen as a special case of the 
magnetohydrodynamical equations \eqref{eq:magneto}
by setting $\f H\equiv 0$.
However, here we can give a finer choice of the regularity weight $\mathcal K$ 
that allows us to show that energy-variational solutions 
are also dissipative weak solutions. 

The  incompressible Euler equations are given by
\begin{subequations}\label{eq:Incomp}
\begin{alignat}{2}
\t \f v + ( \f v \cdot \nabla ) \f v + \nabla p = \f 0 , \quad \di \f v ={}& 0 \qquad && \text{in }\T \times (0,T)\,,
\\
\f v (0) ={}& \f v_0 \qquad && \text{in } \T \,.
\end{alignat}
\end{subequations}
Again, $ \f v : \T\times (0,T) \ra \R^d $ denotes the velocity of the fluid and $p : \T\times (0,T) \ra \R$ denotes the pressure. 
We introduce the energy  $\E : \LRsigma{2}(\T) \ra \R$ 
with $\E(\f v) := \frac{1}{2} \norml{\f v }_{L^2(\T)}^2$.

\begin{definition}\label{def:envarIncommp}
A pair $(\f v , E )\in  L^\infty(0,T; L^2_\sigma ( \T ))  \times \BV $  is called an energy-variational solution to the incompressible Euler system 
\eqref{eq:Incomp}
if $E(t) \geq \E (\f v )$ for a.e.~$t\in (0,T)$
and if the inequality
\begin{equation}
\left [ E - \int_{\T} 
 \f v  \cdot \f \varphi 
 \de \f x  \right ] \Big|_s^t 
+ \int_s^t \int_{\T}
\f v \cdot \t \f \varphi 
+ { \f v  \otimes \f v}  : \nabla \f \varphi 
\de \f x 
+ \mathcal{K}(\f \varphi ) \left [ \E(\f v) - E\right ] 
\de  s
\leq 0
\,\label{relenIncomp}
\end{equation}
holds for a.e.~$s<t\in(0,T)$, including $s=0$ with $\f v (0)  =  \f v_0 $,
and for all test functions $\f \varphi\in \C^1(\T \times [0,T]; \R^{d})$ with $ \di \f \varphi = 0$,
where
\begin{equation}\label{eq:regweight.incomp}
\mathcal{K}(\f \varphi) =2
\|(\nabla \f \varphi )_{\sym,-}\| _{L^\infty(\O;\R^{d\times d})}\,.
\end{equation}
\end{definition} 
Besides existence of energy-variational solutions,
we shall show that they can be identified with so-called weak dissipative solutions
to the incompressible Euler equations \eqref{eq:Incomp}. 
The following definition is an adaption of the compressible case, 
see Definition~\ref{def:weakEul} below.

 \begin{definition}[Dissipative weak solution]\label{def:disssol.incomp}
We call a pair   $(\f v , E )\in  L^\infty(0,T; L^2_\sigma ( \T ))  \times \BV $
a dissipative weak solution to the Euler equations, if there exists a \textit{Reynolds defect}  $ \reynold \in  L^\infty _{w^*} (0,T;\mathcal{M}( \O ; \mathbb{R}^{d\times d}_{\sym,+}))$ such that the equation 
\begin{align}
\int_{\T} \f v \cdot  \f \varphi\de \f x \Big|_s^t
=  \int_s^t\!\! \int_{\T} \f v \cdot \partial_t \f \varphi  +  \f v \otimes \f v:  \nabla \f \varphi\de \f x\de s   + \int_s^t\!\!\int_{\T}   \nabla \f \varphi:\de  \reynold(s) \de s \label{measeq}
\end{align}
is fulfilled for all $ \f \varphi \in \C^1(\T\times[0,T];\R^d )$ with $\dv\f\varphi=0$,
and for  a.a.~$s,t\in(0,T) $, including $s=0$ with $\f v(0)=\f v_0$,
and if $E$ is a   non-increasing function with $E(0+) = \E(\f v_0)$ such that 
\begin{equation}
\E(\f v (t))+ \frac{1}{2}\int_{\T} I :  \de \reynold(t)   \leq E(t)   \,\label{measeneq}
\end{equation}
for a.a.~$t\in(0,T)$.
\end{definition}

 \begin{theorem}\label{thm:incompEuler}
For every initial datum $\f v _0\in L^2_{\sigma}(\T)$, there is an energy-variational solution in the sense of Definition~\ref{def:envarIncommp}
with $E(0)=\mathcal{E}(\f v _0)$ with $\f v \in \C_{w}([0,T]; L^2_{\sigma}(\T) )$ such that the initial condition is attained in the strong sense. 
Moreover, a pair  $(\f v , E )\in  L^\infty(0,T; L^2_\sigma ( \T ))  \times \BV $ is an energy-variational solution in the sense of Definition~\ref{def:envarIncommp} if and only if it is a  dissipative weak solution in the sense of Definition~\ref{def:disssol.incomp}. 
\end{theorem}
 \begin{proof}
At first, we show that the Hypothesis~\ref{hypo} is fulfilled,
which is very similar to the proof of Theorem~\ref{thm:magneto}. 
To the most extent, we can copy the above proof with $\f H\equiv 0$ or vanishing second component in all functionals. 
But since we assert that
the regularity weight $\mathcal K$ can be chosen in the finer manner
stated in \eqref{eq:regweight.incomp},
it remains to verify the convexity of the function from~\eqref{ass:convex}
with this choice.
Indeed, we have 
 \[
 \begin{aligned}
     &\int_{\T} \f F( \f v ) : \nabla \f \varphi \de \f x 
     + \mathcal K(\f \varphi)\E(\f v)\\
     &= \int_{\T} \f v \otimes \f v : ( \nabla \f \varphi)_{\sym,+} \de \f x
     + \int_{\T} ( \f v \otimes \f v ) : \bb{( \nabla \f \varphi)_{\sym,-} + \| (\nabla \f \varphi )_{\sym,-}\| _{L^\infty(\O;\R^{d\times d})} I} \de \f x \,,
 \end{aligned}
 \]
 where we infer the convexity and weak lower semi-continuity  of both terms in the second line since they are non-negative and quadratic. 
Hence, Hypothesis~\ref{hypo} is satisfied, 
and from Theorem~\ref{thm:main}, we infer the existence of an energy-variational solution. 
 
 Now let~$(\f v ,E)$ be an energy-variational solution in the sense of Definition~\ref{def:envarIncommp}. 
 The choice $ \f \varphi = \f 0$ implies that $ E $ is non-increasing.
 Since the regularity weight $\mathcal K$ is homogeneous of degree one, we infer from Proposition~\ref{prop:equality} that
\begin{align}
- \int_{\T} \f v \cdot \f \psi  \de \f x   \Big|_{0}^T + \int_0^T \int_{\T} \f v \cdot \t \f\psi + ( \f v \otimes \f v ) : \nabla \f \psi \de \f x \de t \leq \int_0^T \mathcal{K}(\f \psi ) [ E - \E(\f v) ] \de t \label{estincomp}
\end{align}
for all $ \f \psi \in \C^1(\T\times[0,T];\R^d)$ with $\dv\f\psi=0$.  
We define
\[
\begin{aligned}
&\mathcal V:= \{\f \varphi\in \C_0^1(\O\times[0,T);\R^d) \mid  \di \f \varphi = 0 \text{ a.e.~in }\T \times (0,T)\,,\  \int_\O\f\varphi\de x =0\}\,,
\\
&\f l \colon\mathcal V\to\R,
\quad
\langle \f l, \f \psi \rangle := - \int_{\T} \f v \cdot \f \psi  \de \f x   \Big|_{0}^T + \int_0^T \int_{\T} \f v \cdot \t \f\psi + ( \f v \otimes \f v ) : \nabla \f \psi \de \f x \de t\,,
\\
&\mathfrak p\colon L^1(0,T;\C(\O ; \R^{d\times d}_{\sym}))  \ra \R,
\quad
\mathfrak p(\Phi) := \int_0^T 2\| (\Phi)_{-}\|_{\C(\T;\R^{d\times d})}(E - \mathcal{E}(\f v))\de t\,.
\end{aligned}
\]
Due to \eqref{estincomp}, Lemma~\ref{lem:hahn} implies 
that there exists
$\mathfrak R \in L^\infty_{w^*}(0,T;\mathcal{M}(\O ; \R_{\sym}^{d\times d}
))$ 
with
\[
\forall\Phi\in L^1(0,T; \C(\O; \R_{\sym}^{d\times d })):\ \langle -\mathfrak R, \Phi\rangle \leq \mathfrak p(\Phi),
\qquad
\forall\f \psi \in  \mathcal{V} : \ \langle -\mathfrak R, \nabla \f \psi \rangle = \langle \f l , \f \psi \rangle.
\]
The first property implies $\langle \mathfrak R, \Phi\rangle \geq 0$ 
if $\Phi$ is positive semi-definite  in $\T\times (0,T)$, 
so that we have $\mathfrak R \in L^\infty_{w^*}(0,T;\mathcal{M}(\O ; \R_{\sym,+}^{d\times d}))$. 
The second property yields \eqref{measeq} 
for $\f \psi\in \mathcal V$.
Using $\f \psi=\f e_j$ in~\eqref{estincomp},
where $\f e_j$ is the $j$-th unit vector in $\R^d$,
we see that $\int_{\O}\f v\de x$ is constant in time.
Therefore, we can drop the mean-value condition on $\f \psi$
and infer \eqref{measeq} for all $\f \psi \in \C^1(\O\times[0,T]; \R ^d )$
with $\dv\f\psi=0$.
Considering
$\Phi(x,t)= - \phi(t)I$ for some $\phi\in\C_0^{1}([0,T))$ with $\phi \geq 0$,
we further have
\[
\int_0^T \phi(t) \int_{\T} I :  \de \reynold(t) \de t
=\langle -\mathfrak R, \Phi\rangle 
\leq \mathfrak p(\Phi)
=2\int_0^T \phi(t) (E - \mathcal{E}(\f v))\de t.
\]
Since $\phi\geq 0$ is arbitrary,  this directly implies \eqref{measeneq} 
for a.a.~$t\in(0,T)$.
In total, we see that $(\f v, E)$ is a dissipative weak solution.

In order to prove the converse implication,  
let  $(\f v , E )\in  L^\infty(0,T; L^2_\sigma ( \T ))  \times \BV $ 
be a dissipative weak solution to \eqref{eq:Incomp}. 
Due to $\mathfrak{R}(t) \in \mathcal{M}(\T;\R^{d\times d}_{\sym,+})$, the duality of the spectral norm and the trace norm for matrices,
H\"older's inequality and inequality~\eqref{measeneq} allow to infer 
\begin{align*}
\int_{\T}  \nabla \f \psi  : \de \mathfrak{R} \geq \int_{\T} ( \nabla \f \psi )_{\sym,-} : \de \mathfrak{R} 
& \geq - \|  ( \nabla \f \psi )_{\sym,-} \|_{L^\infty(\T;\R^{d\times d })} 
\int_{\T}I : \de \mathfrak{R}
\\
& \geq 2 \|  ( \nabla \f \psi )_{\sym,-} \|_{L^\infty(\T;\R^{d\times d })} \left [  \E(\f v ) -E \right ] \,
\end{align*}
a.e.~in $(0,T)$.
Estimating the last term of \eqref{measeq} with $\f\varphi=-\f \psi$ in this way,
we obtain
\[
-\left [ \int_{\T} 
 \f v  \cdot \f \psi
 \de \f x  \right ] \Big|_s^t 
+ \int_s^t \int_{\T}
\f v \cdot \t \f \psi 
+ { \f v  \otimes \f v}  : \nabla \f \psi 
\de \f x 
+ \mathcal{K}(\f \psi) \left [ \E(\f v) - E\right ] 
\de  s
\leq 0
\,.
\]
Since $E$ is non-increasing, 
we may add the term $E\big|_s^t$ to the left-hand side
to infer the formulation~\eqref{relenIncomp}. 
 \end{proof}
 
\begin{remark}[Trace-free measures]
Due to the fact that the equation~\eqref{measeq} holds for solenoidal test functions, one may change the measure $\mathfrak{R}$ in this formulation by adding a multiplicative of the identity. 
This can be done in such a way that the resulting measure $\bar{\mathfrak{R}}$ is trace-free 
by setting $ \bar{\mathfrak{R}}= \mathfrak{R}-\frac{1}{d} \tr(\mathfrak{R})I$. 
Consequently, we could adapt Definition~\ref{def:disssol.incomp} by requiring $ \mathfrak R\in L^\infty_{w}(0,T;\mathcal{M}(\T ; \R^{d\times d}_{\sym,0}))$, where $\R^{d\times d}_{\sym,0} $ denotes the set of symmetric trace-free matrices,
and by demanding the simpler inequality $\E(\f v) \leq E $
instead of inequality~\eqref{measeneq}. 
We could infer this formulation with the same arguments as above, but by choosing $\mathcal{I}(\f \psi) = (\nabla \f \psi)_{\sym} - \frac{1}{d} \tr (\nabla \f \psi ) I$ in Lemma~\ref{lem:hahn}  and  $ \mathfrak p(\Phi) = \int_0^T 2\| \Phi\|_{\C(\T;\R^{d\times d})}(E - \mathcal{E}(\f v))\de t$. 
However, we prefer the choice made in Definition~\ref{def:disssol.incomp} since in inequality~\eqref{measeneq}  the dissipative nature of the Reynolds defect $\mathfrak{R}$ becomes visible. 
\end{remark}

\section{Compressible Euler equations\label{sec:comp}}
\label{sec:comprEuler}
Now, we turn to the compressible Euler system.
Here, instead of formulating the equations
in terms of the density $\dens$ and the fluid velocity $\f v$, 
we use the density and the momentum $\mom=\dens\f v$. 
This is often done in the literature, see for instance~\cite{Fereisl21_NoteLongTimeBehaviorDissSolEuler}.
The main reason for this choice is 
that the associated energy functional is convex in the variables $(\dens,\mom)$
as we will see below.
The Euler equations then read 
\begin{subequations}
\label{eq:comprEuler}
\begin{alignat}{2}
\t \dens + \di  \mom   ={}& 0 &&\quad\text{in  }\T\times (0,T),
\label{eq:comprEuler.mass}
\\
\t \mom   + \di \left ( \frac{\mom \otimes \mom}{\dens} \right ) + \nabla \pres(\dens) ={}& 0 
&&\quad \text{in  } \T \times (0,T),
\label{eq:comprEuler.momentum}
\\
(\dens,\mom)(\cdot,0)={}&(\dens_0,\mom_0)
&&\quad\text{in }\T.
\end{alignat}
\end{subequations}
Here $\dens\colon\T\times(0,T)\to[0,\infty)$ and
and $\mom\colon\T\times(0,T)\to\R^d$
denote the mass density and the momentum field of an inviscid fluid flow,
and the pressure $\pres$ is related to the density $\dens$ by 
a barotropic pressure law $\pres=\pres(\dens)$.
Note that we follow~\cite{weakstrongCompEul} 
and use $h$ for the density variable instead of $\rho $,
which fits to our notation 
to use Latin letters for the state variables 
and Greek letters for the test functions. 

To see that \eqref{eq:comprEuler} belongs to the class of hyperbolic conservation laws
introduced above,
we set
\[
\f F ( \dens,\mom ) =
\begin{pmatrix}
 \mom ^T  \\ 
 \bp{\frac{\mom \otimes \mom }{\dens}  + p(\dens) I}\chi_{(0,\infty)}(\dens)
\end{pmatrix}.
\] 
Then \eqref{eq:comprEuler} is equivalent to \eqref{eq} with 
$ \f U = (\dens , \mom) $. 
The mathematical entropy $\eta$ for the system 
is defined as
\[
\eta (\dens,\mom) 
=\begin{cases} \frac{1}{2}\frac{|\mom |^2}{\dens} 
+ \pot(\dens)
& \text{if }h> 0,\\
0
& \text{if }(h, \f m) = (0,\f 0) ,
\\
\infty & \text{else},
\end{cases}
\]
and 
$\E(\dens,\mom)=\int_\T\eta(\dens(x),\mom(x))\,\dx$
is the total physical energy.
Here $\pot$ denotes the potential energy, which is associated to the 
pressure $\pres$ via
\begin{equation}\label{eq:pot.from.pres}
\pot(\dens)=\dens\int_0^\dens \frac{\pres(z)}{z^2}\de z.
\end{equation}
Vice versa, the pressure $\pres$ can be derived from the potential energy $\pot$ via
\begin{equation}\label{eq:pres.from.pot}
\pres(\dens)=\dens\pot'(h)-\pot(h).
\end{equation}
For conditions ensuring that all expressions in \eqref{eq:pot.from.pres} and
\eqref{eq:pres.from.pot} are well defined,
we refer to \eqref{eq:pres.properties} and 
\eqref{eq:pot.properties} below, respectively.

\subsection{Energy-variational solutions to the compressible Euler equations}
For the sake of convenience, we now transfer Definition~\ref{def:envar} 
to the compressible Euler system,
and express all quantities in the way considered here.
\begin{definition}\label{def:envarEUL*}
A triple 
$(\dens,\mom, E)\in 
L^1_{\mathrm{loc}}(\T\times(0,T);[0,\infty))\times 
L^1_{\mathrm{loc}}(\T\times(0,T);\R^d)
\times \BV$
is called an energy-variational solution to the compressible Euler system 
\eqref{eq:comprEuler} 
if 
$\E (\dens(t) , \mom(t))\leq E(t)$ for a.e.~$t\in (0,T)$
and
if 
the energy-variational inequality
\begin{multline}
\left [ E - \int_{\T} h \rho 
+ \f m  \cdot \f \varphi 
 \de \f x  \right ] \Big|_s^t + \int_s^t 
\int_{\T} h \t \rho
+\f m \cdot  \nabla  \rho 
\de \f x 
\de \tau
\\
+ \int_s^t \int_{\T}
\f m \cdot \t \f \varphi 
+  \left (\frac{ \f m  \otimes \f m}{h} + \pres(\dens)I \right ): \nabla \f \varphi 
\de \f x 
+ \mathcal{K}_\alpha (\rho,\f \varphi ) \left [ \E(h ,\f m) - E\right ] 
\de  \tau
\leq 0
\,\label{relenEul*}
\end{multline}
holds for a.e.~$s<t\in(0,T)$, including $s=0$ with $( \dens (0), \mom (0)) = ( \dens_0 ,\mom_0)$, 
and for all test functions 
$(\rho,\f \varphi )\in \C^1(\T\times [0,T])\times\C^1(\T\times [0,T];\R^m)$,
where $\mathcal K_\alpha $ is given by
\begin{align}
\mathcal{K}_\alpha (\denst,\momt )=\mathcal{K}_\alpha(\momt) =
\max\setl{ 2,\alpha d}
\| (\nabla \f \varphi )_{\sym,-}\| _{L^\infty(\O;\R^{d\times d})}
\label{K:CompEul}\,
\end{align}
for a suitable choice of $\alpha>0 $. 
\end{definition}
\begin{remark}[Choice of regularity weight]
There are different choices possible for the regularity weight~$\mathcal{K}_\alpha$. A finer choice would be given by 
\begin{multline*}
    \tilde{\mathcal{K}}_{\alpha}(\varphi) = \max\setl{ 2\| (\nabla \f \varphi )_{\sym,-}\| _{L^\infty(\O;\R^{d\times d})},\alpha \| (\di \varphi )_{-}\| _{L^\infty(\O;\R)}} 
\\
\leq \max\setl{ 2,\alpha d}
\| (\nabla \f \varphi )_{\sym,-}\| _{L^\infty(\O;\R^{d\times d})} =\mathcal{K}_\alpha(\f\varphi)\,. \end{multline*}
Note that the solution concept is finer for a smaller regularity weight, since the energy-variational inequality~\eqref{envarform} remains valid, if the regularity weight increases (\textit{cf.}~\cite[Prop.~4.4]{EiHoLa22}). 
Nevertheless, we use the  above choice
since it yields the equivalence to dissipative weak solutions; see Theorem~\ref{thm:measval.CompEul} below. 
\end{remark}

To show existence of
 energy-variation solutions to \eqref{eq:comprEuler},
we restrict the class of 
admissible pressure laws and assume that $\pres$ is of the form
$\pres(\dens)=a\dens^\gamma$.
Then we show the following result.

\begin{theorem}\label{thm:exCompEul}
Let $\pres(\dens)=a\dens^\gamma$ for some $a>0$, $\gamma>1$,
and set $q=2\gamma/(1+\gamma)$ and $\alpha =\gamma-1$.
For every initial data $ ( h_0, \f m_0)\in L^1_{\mathrm{loc}}(\T;\R^{d+1})$
with $\E( h_0, \f m_0)<\infty$
there exists an energy-variational solution
\[
(\dens,\mom,E)\in \C_w([0,T]; L^\gamma ( \T )) \times \C_w([0,T]; L^q (\T ;\R^d )) \times \BV
\]
to the compressible Euler equations~\eqref{eq:comprEuler} in the sense of Definition~\ref{def:envarEUL*} with
 $E(0)=\mathcal{E}(\dens_0, \mom_0)$ and
 such that the initial conditions are attained in the strong sense. 
\end{theorem}

\subsection{Existence of energy-variational solutions}

To prove existence of an energy-variational solution,
we show that all assertions
of Theorem \ref{thm:main} are satisfied.
Actually, most of them
can be shown for more general pressure laws 
than those in the statement of Theorem \ref{thm:exCompEul}.
For the moment, we shall merely assume that the potential energy $\pot$
satisfies
\begin{subequations}\label{eq:pot.properties}
\begin{align}
\pot\in\C^1[0,\infty)\cap\C^2(0,\infty),
\qquad
&\pot''(z)>0 \text{ for all } z>0,
\label{eq:pot.regularity}
\\
\lim_{z\to\infty} \frac{\pot(z)}{z}=\infty,
\qquad
&\pot(0)=\pot'(0)=0.
\label{eq:pot.limits}
\end{align}
\end{subequations}
In particular, $\pot$ is a strictly convex function with superlinear growth,
and $\pres$ is well defined via \eqref{eq:pres.from.pot}.

\begin{remark}
\label{rem:pressurelaw}
In \eqref{eq:pot.properties}
we introduced assumptions on the potential energy $\pot$,
while in the literature it is much more common to 
state assumptions on the pressure $\pres$ directly.
To guarantee \eqref{eq:pot.regularity} and \eqref{eq:pot.limits},
one may assume that $\pres$ satisfies
\begin{subequations}\label{eq:pres.properties}
\begin{align}
\pres\in\C^0[0,\infty)\cap\C^1(0,\infty),
\qquad
&\pres'(z)>0 \text{ for all } z>0, 
\label{eq:pres.regularity}
\\
\lim_{z\to\infty} \frac{\pres(z)}{z}=\infty,
\qquad
&\int_0^1 \frac{\pres(z)}{z^2}\,\de z <\infty.
\label{eq:pres.limits}
\end{align}
\end{subequations}
One readily verifies that then the right-hand side of 
\eqref{eq:pot.from.pres} is well defined, 
and that \eqref{eq:pot.regularity} and \eqref{eq:pres.regularity}
are equivalent since $\pot''(z)=\pres'(z)/z$.
Moreover, the second condition in \eqref{eq:pres.limits}
implies $\lim_{z\to 0} \pres(z)/z =0$,
and it is equivalent to the second condition in \eqref{eq:pot.limits},
which follows with the identity
\[
\int_0^1 \frac{\pres(z)}{z^2}\de z
=\int_0^1\ddz\bb{\frac{\pot(z)}{z}}\de z
=\pot(1)-\lim_{z\to0} \frac{\pot(z)}{z}
=\pot(1)-\pot'(0). 
\]
Additionally, superlinear growth of $\pres$ implies superlinear growth of $\pot$.
Indeed, the first condition in \eqref{eq:pot.limits} yields the existence of $\dens_0>0$ 
such that $\pres(z)\geq z$ for all $z\geq\dens_0$,
whence we have
\[
\frac{\pot(\dens)}{\dens}
\geq \int_0^{\dens_0} \frac{\pres(z)}{z^2}\de z
+\int_{\dens_0}^\dens \frac{1}{z}\de z
=\int_0^{\dens_0} \frac{\pres(z)}{z^2}\de z
+\log(\dens)-\log(\dens_0)
\to\infty
\]
as $h\to\infty$.
However, the converse is not true. 
For example, the function $\pot(\dens)=(1+\dens)\log(1+\dens)-\dens$
satisfies \eqref{eq:pot.regularity} and \eqref{eq:pot.limits},
but the associated pressure $\pres(\dens)=\dens-\log(1+\dens)$ does not 
have superlinear growth.
Therefore, the assumption 
on the potential $\pot$ in \eqref{eq:pot.properties}
are less restrictive
than the assumptions
on the pressure $\pres$ in \eqref{eq:pres.properties}, 
which explains why we work with the former in what follows.
\end{remark}

We separate the proof into several lemmas,
the first one concerns properties of $\eta$.

\begin{lemma}\label{lem:eta.comprEuler}
If $\pot$ satisfies \eqref{eq:pot.properties},
then $\eta$ is strictly convex and has superlinear growth.
\end{lemma}
\begin{proof}
By \eqref{eq:pot.regularity}, we have $\pot''>0$,
and the strict convexity of $\eta$ directly follows by computing the second derivatives.
To show the superlinear growth, let $(\dens_n,\mom_n)\subset\R\times\R^n$ be a sequence 
with $\snorm{(\dens_n,\mom_n)}\to\infty$.
If $\dens_n\to-\infty$ as $n\to\infty$,
then we clearly have $\eta(\dens_n,\mom_n)/\snorm{(\dens_n,\mom_n)}\to\infty$.
So we may assume $\dens_n>0$ in the following.
If $\lim_{n\to\infty}\snorm{\mom_n}/\dens_n= \infty$,
then
\[
\frac{\eta(\dens_n,\mom_n)}{\snorm{\np{\dens_n,\mom_n}}}
\geq \frac{\snorm{\mom_n}^2}{2\dens_n\,\snorm{\np{\dens_n,\mom_n}}}
=\frac{\snorm{\mom_n}}{\sqrt{\dens_n^2+\snorm{\mom_n}^2}}\frac{\snorm{\mom_n}}{2\dens_n} 
\to \infty
\]
as $n\to\infty$;
if
$\liminf_{n\to\infty}\snorm{\mom_n}/\dens_n= c \geq 0$,
then we have $h_n\to\infty$ and
\[
\frac{\eta(\dens_n,\mom_n)}{\snorm{\np{\dens_n,\mom_n}}}
\geq \frac{P(\dens_n)}{\snorm{\np{\dens_n,\mom_n}}}
=\frac{\dens_n}{\sqrt{\dens_n^2+\snorm{\mom_n}^2}} 
\frac{\pot(h_n)}{h_n} \to \infty
\]
as $n\to\infty$ due to \eqref{eq:pot.limits}.
In total, this completes the proof.
\end{proof}

Next we calculate the convex conjugate $\eta^*$ of $\eta$.
To this end, we use that $P'$ is an invertible mapping.

\begin{lemma}\label{lem:etastar.comprEuler}
If $\pot$ satisfies \eqref{eq:pot.properties}, 
then the mapping $Q\coloneqq P'$ is strictly increasing
and a bijective self-mapping on $[0,\infty)$.
The convex conjugate $\eta^*$ of $\eta$ is given by
\begin{equation}\label{eq:etastar.comprEuler}
\eta^\ast ( \denst , \momt  )
=\int_0^{\bp{\denst+\frac{\snorm{\momt}^2}{2}}_+} Q^{-1}(z)\,\dz
=\pres\circ Q^{-1}\Bp{\Bp{\denst+\frac{\snorm{\momt}^2}{2}}_+},
\end{equation}
and it holds
\begin{equation}\label{eq:Detastar.comprEuler}
D\eta^\ast ( \denst , \momt  )
= Q^{-1}\Bp{\Bp{\denst+\frac{\snorm{\momt}^2}{2}}_+}\begin{pmatrix}
1\\
\f \varphi 
\end{pmatrix}
\end{equation}
where $z_+\coloneqq\max\set{z,0}$ for $z\in\R$.
\end{lemma}

\begin{proof}
By \eqref{eq:pot.properties},
the function $Q=\pot'$ is strictly increasing and continuous on $[0,\infty)$
with $Q(0)=0$,
and the convexity of $\pot$ yields
\[
Q(z)=\pot'(z)\geq\frac{\pot(z)}{z}\to \infty
\]
as $z\to\infty$.
Therefore, $Q$ is a bijective self-mapping on $[0,\infty)$
with inverse $Q^{-1}$.  
The second equality in \eqref{eq:etastar.comprEuler} is now
 a direct consequence of the identity
\[
\ddz p(Q^{-1}(z))
= p'(Q^{-1}(z)) (Q^{-1})'(z)
= p'(Q^{-1}(z)) \frac{1}{Q'(Q^{-1}(z))}
= Q^{-1}(z),
\]
where we used $Q'(z)=\pot''(z)=\pres'(z)/z$.
To verify the first equality in \eqref{eq:etastar.comprEuler},
consider the case $\denst+\frac{\snorm{\momt}^2}{2}\leq0$ at first.
We employ Young's inequality to estimate
\[
\dens\,\denst + \mom\cdot\momt - \eta(\dens,\mom)
\leq -\dens\frac{\snorm{\momt}^2}{2} + \frac{\snorm{\mom}^2}{2\dens} 
+ \frac{\dens\snorm{\momt}^2}{2} - \eta(\dens,\mom)
\leq 0,
\]
which shows $\eta^\ast ( \rho , \f \varphi  ) \leq 0$.
Since we also have $\eta^\ast\geq 0$, 
we infer $\eta^\ast ( \rho , \f \varphi  )=0$,
which is \eqref{eq:etastar.comprEuler} 
if $\denst+\frac{\snorm{\momt}^2}{2}\leq0$.
This also implies \eqref{eq:Detastar.comprEuler} in this case.
If $\denst+\frac{\snorm{\momt}^2}{2}>0$, then $(\denst,\momt)$ 
belongs to the range of $D\eta$, which is given by
\[
D \eta (\dens,\mom) 
= \left( -\frac{\snorm{\mom}^2}{2\dens^2}+Q(\dens), \frac{\mom}{\dens} \right)
\]
for $\dens>0$, $\mom\in\R^d$.
Computing the inverse, we arrive at \eqref{eq:Detastar.comprEuler}
in this case.
This further 
yields \eqref{eq:etastar.comprEuler} for $\denst+\frac{\snorm{\momt}^2}{2}>0$
since $Q(0)=0$.
In total, we have thus verified \eqref{eq:etastar.comprEuler}
and \eqref{eq:Detastar.comprEuler}.
\end{proof}

In the next lemma, we verify the compatibility condition \eqref{eq:integralFentropy}
between $F$ and the entropy functional $\eta$.

\begin{lemma}\label{lem:integralcond.comprEuler}
If $\pot$ satisfies \eqref{eq:pot.properties}, 
then the condition \eqref{eq:integralFentropy} is satisfied.
\end{lemma}
\begin{proof}
We first consider the integrand of \eqref{eq:integralFentropy}.
Using \eqref{eq:Detastar.comprEuler}, 
for all $(\denst,\momt)\in\C^1(\T;\R^{d+1})$ we have
\[
\begin{aligned}
\f F(&D\eta^\ast(\denst,\momt)):\nabla
\begin{pmatrix} 
\denst \\ 
\momt 
\end{pmatrix}
= \begin{pmatrix} 
\momt^T \ Q^{-1}\bp{\bp{\denst+\frac{\snorm{\momt}^2}{2}}_+} \\ 
\momt\otimes\momt \ Q^{-1}\bp{\bp{\denst+\frac{\snorm{\momt}^2}{2}}_+}
+ p\circ Q^{-1}\bp{\bp{\denst+\frac{\snorm{\momt}^2}{2}}_+}\, I
\end{pmatrix}
:\nabla
\begin{pmatrix} 
\denst \\ 
\momt 
\end{pmatrix}
\\
&= \momt\cdot\nabla\Bp{\denst+\frac{\snorm{\momt}^2}{2}} \
Q^{-1}\Bp{\Bp{\denst+\frac{\snorm{\momt}^2}{2}}_+} 
+ p\circ Q^{-1}\Bp{\Bp{\denst+\frac{\snorm{\momt}^2}{2}}_+} \dv\momt
\\
&=\dv \bb{\momt \,\eta^\ast(\denst,\momt)},
\end{aligned}
\]
where the last equality follows from \eqref{eq:etastar.comprEuler}.
Integrating this identity
yields \eqref{eq:integralFentropy}.
\end{proof}

\begin{remark}\label{rem:entropyflux}
We can also show \eqref{eq:integralFentropy} 
by verifying the alternative condition \eqref{eq:entropyflux.new}.
Indeed, we can use \eqref{eq:etastar.comprEuler} and 
\eqref{eq:Detastar.comprEuler} to derive
\[
\begin{aligned}
\f F(D\eta^\ast(\denst,\momt))
&= \begin{pmatrix} 
\momt^T \ Q^{-1}\bp{\bp{\denst+\frac{\snorm{\momt}^2}{2}}_+} \\ 
\momt\otimes\momt \ Q^{-1}\bp{\bp{\denst+\frac{\snorm{\momt}^2}{2}}_+}
+ p\circ Q^{-1}\bp{\bp{\denst+\frac{\snorm{\momt}^2}{2}}_+}\, I
\end{pmatrix}
\\
&= \begin{pmatrix} 
\momt^T \ D_\denst\eta^\ast(\denst,\momt) \\ 
\momt\otimes D_\momt\eta^\ast(\denst,\momt)
+ \eta^\ast(\denst,\momt)\, I
\end{pmatrix}\\
&=D \bb{
\momt \,\eta^\ast(\denst,\momt) }
= D \bb{\momt\, p (D_\denst\eta^\ast(\denst,\momt))}
= D \bb{\tilde{\f q}\circ D\eta^\ast}(\denst,\momt)
\end{aligned}
\]
for $\tilde{\f q}(\dens,\mom)=\frac{\mom}{\dens} \pres(\dens)$.
This shows \eqref{eq:entropyflux.new},
which implies \eqref{eq:integralFentropy} by Remark \ref{rem:integralcondition}.
However,
we cannot use the classical entropy-flux condition \eqref{eq:entropypair}
in the present situation,
since $\f F$ is not differentiable.
\end{remark}

It remains to show that 
the function $\mathcal K$ defined in \eqref{K:CompEul} is a suitable choice.
To show this, we have to impose more restrictive conditions 
on $\pres$ and $\pot$.

\begin{lemma}\label{lem:nonlin.convex.compEuler}
Assume that additionally to \eqref{eq:pot.properties}
there exist constants $c,\alpha>0$ such that
\begin{equation}\label{eq:est.pres.by.pot}
\pres(\dens)\leq c(1+\pot(\dens)), 
\end{equation}
for all $h\geq 0$,
and such that the functions $\pres$ and $\alpha\pot-\pres$
are convex and non-negative.
Then \eqref{BoundF} holds for some $C>0$,
and
for any $\momt\in\C^1(\T;\R^d)$
the mapping
\begin{equation}
\dom\E\to\R,\quad
(\dens,\mom)\mapsto \int_{\T}
 \left (\frac{ \mom \otimes \mom}{\dens} + \pres(\dens)I \right ): \nabla \momt 
\de \f x 
+ \mathcal{K}_\alpha(\momt )\E(\dens ,\mom) \label{mapping}
\end{equation}
is convex, lower semi-continuous and non-negative
for $\mathcal{K}_{\alpha}$ as in \eqref{K:CompEul}.
\end{lemma}

\begin{proof}
For $h\leq 0$, estimate \eqref{BoundF} is trivial,
and for $h>0$ we use Young's inequality and \eqref{eq:est.pres.by.pot} to conclude
\[
\snorm{\f F(\dens,\mom)}
\leq C\Bp{\frac{\snorm{\mom}}{\sqrt{\dens}}\sqrt{\dens}+\frac{\snorm{\mom}^2}{h}+\pres(\dens)}
\leq C\Bp{h+\frac{\snorm{\mom}^2}{h}+1+\pot(\dens)}
\leq C\bp{1+\eta(\dens,\mom)},
\]
where we used the superlinear growth of $\pot$ in the last estimate. 
In total, this shows \eqref{BoundF}.
To deduce that \eqref{mapping} is convex and non-negative, 
firstly note that
the mapping
\[
(h,\f m)\mapsto
\int_{\T}\frac{ \f m  \otimes \f m}{h} {:} \nabla \f \varphi 
\de \f x 
+ \mathcal{K}_{\alpha}(\f \varphi )\int_\T\frac{| \f m|^2}{2h}\,\dx
=\int_{\T}\frac{ \f m  \otimes \f m}{h} {:} ((\nabla \f \varphi)_\sym 
+ \frac{1}{2} \mathcal{K}_{\alpha}(\momt )I )\de \f x 
\]
is convex and non-negative because the matrix
\[
(\nabla \f \varphi)_\sym 
+ \frac{1}{2} \mathcal{K}_{\alpha}(\momt )I 
= (\nabla \f \varphi)_{\sym,+} + (\nabla \f \varphi )_{\sym,-}
+ \frac{1}{2} \mathcal{K}_{\alpha}(\momt )I 
\]
is symmetric and positive semi-definite.
For the term $(\nabla \f \varphi)_{\mathrm{sym},+}$ this is clear,
and for the remaining term this follows from 
$\frac{1}{2}\mathcal{K}_\alpha(\f \varphi ) 
\geq \| (\nabla \f \varphi )_{\sym,-}\| _{L^\infty(\O;\R^{d\times d})}$.
Secondly, the mapping
\[
\begin{aligned}
\dens\mapsto
\int_{\T}
&{}\pres(\dens) I : \nabla \momt 
\de \f x 
+ \mathcal{K}_{\alpha}(\momt )\int_{\T} \pot(\dens)\,\de \f x
\\
&=\int_{\T}
\pres(\dens) I:(\nabla\momt)_{\sym,+}
+ \frac{1}{d}\int_{\T} \pot(\dens)I:
\bp{\mathcal{K}_{\alpha}(\momt)I+\alpha d(\nabla\momt)_{\sym,-}}\,\de \f x
\\
&\qquad+ \int_{\T} \bp{\pres(\dens)-\alpha\pot(\dens)}I:(\nabla\momt)_{\sym,-}\de \f x
\end{aligned}
\]
is also convex and non-negative.
Indeed, this follows from the convexity of $\pres$, $\pot$ 
and $\alpha\pot-\pres$ and from 
the fact that $(\nabla\momt)_{\sym,+}$,
$-(\nabla\momt)_{\sym,-}$
and $\mathcal{K}_{\alpha}(\momt)I+\alpha d(\nabla\momt)_{\sym,-}$ 
are positive semi-definite in $\T$.
Note that for the last term,
this follows from 
$\mathcal{K}_\alpha(\f \varphi ) 
\geq \alpha d\| (\nabla \f \varphi )_{\sym,-}\| _{L^\infty(\O;\R^{d\times d})}$.
In total, the asserted convexity and non-negativity of \eqref{mapping} follows. 
Finally, since strong convergence 
implies point-wise convergence almost everywhere of a subsequence,
the non-negativity of the mapping~\eqref{mapping} and Fatou's lemma 
imply the lower semi-continuity of~\eqref{mapping}. 
\end{proof}

Finally, we prove Theorem \ref{thm:exCompEul} on 
existence of energy-variational solutions to the compressible Euler system.

\begin{proof}[Proof of Theorem \ref{thm:exCompEul}]
If $\pres(\dens)=a\dens^\gamma$, 
then $\pot(\dens)=(\gamma-1)^{-1}a\dens^\gamma$,
and one directly sees that all properties from \eqref{eq:pot.properties} 
(or even \eqref{eq:pres.properties}) 
are satisfied, and Lemma \ref{lem:eta.comprEuler} and
Lemma \ref{lem:integralcond.comprEuler} are applicable.
Moreover, we have $\pres(\dens)=(\gamma-1)\pot(\dens)$,
so that the assumptions of Lemma \ref{lem:nonlin.convex.compEuler}
are satisfied with $c=\alpha=\gamma-1$.
Moreover, we may identify $\domE = \dom \E$, which  is convex, \textit{cf.} Remark~\ref{rem:domE}. 
From Theorem \ref{thm:main} we thus conclude the existence of 
energy-variational solutions $(\dens,\mom)$ 
in the sense of Definition~\ref{def:envarEUL*}.
Moreover,
Young's inequality implies
\[ 
|\f m| ^q= \left (\frac{|\f m|}{\sqrt{h}}\right )^q {h}^{\frac{q}{2}}  \leq \frac{q}{2} \frac{|\f m| ^2 }{h} + \frac{2-q}{2} h^{\frac{q}{2-q}} = \frac{\gamma}{1+\gamma} \frac{|\f m| ^2 }{h} + \frac{1}{1+\gamma}  h ^\gamma  \,,
\]
whence
\[
\int_\O \dens(t)^\gamma+\snorm{\mom(t)}^q \,\dx
\leq 
C\,\E(\dens(t),\mom(t))
\leq C\,
\E(\dens_0,\mom_0)
\]
for a.a.~$t\in(0,T)$ and some $C=C(\gamma)>0$.
Finally, the assumptions of Proposition~\ref{prop:betterreg}
are clearly satisfied
such that 
$(\dens,\mom)$ belongs to the asserted function class.
\end{proof}

\begin{remark}
It is readily seen that the previous proof also works for more general
pressure laws than the above choice $\pres(\dens)=a\dens^\gamma$
since it suffices to satisfy condition \eqref{eq:pot.properties} and 
the assumptions from Lemma \ref{lem:nonlin.convex.compEuler}
to obtain existence.
For example,  
one may consider 
pressure laws of the form $\pres(\dens)=a_1\dens^{\gamma_1}+a_2\dens^{\gamma_2}$
with $a_1,a_2>0$ and $\gamma_1,\gamma_2>1$.
One easily checks that then \eqref{eq:pot.properties}
is satisfied, 
and the assumptions of Lemma \ref{lem:nonlin.convex.compEuler}
hold with $c=\alpha=\max\set{\gamma_1,\gamma_2}-1$.
Another example would be the pressure law $\pres(\dens)=\dens-\log(1+\dens)$
with associated
potential energy $\pot(\dens)=(1+\dens)\log(1+\dens)-\dens$
from Remark \ref{rem:pressurelaw},
where one can choose $c=\alpha=1$.
\end{remark}

\subsection{Comparison with dissipative weak solutions}

To compare the notion of energy-variational solutions with 
existing solution concepts for the compressible Euler system \eqref{eq:comprEuler}, 
we recall the notion of dissipative weak solutions 
for pressure laws $\pres(\dens)=a\dens^\gamma$
(cf.~\cite[Def.~2.1]{Fereisl21_NoteLongTimeBehaviorDissSolEuler}).

\begin{definition}\label{def:weakEul}
We call a tuple $(h,\f m, E )\in L^\infty(0,T; L^\gamma ( \T )) \times L^\infty(0,T; L^q (\T ;\R^d ))  \times \BV$  with $q=2\gamma/(1+\gamma)$ a dissipative weak solution to the compressible Euler system
\eqref{eq:comprEuler} if 
there exists a so-called \emph{Reynolds defect}
$\mathfrak R\in L^\infty_{w^*}(0,T;\mathcal{M}(\T ; \R^{d\times d}_{\sym,+}))$
such that the equations
\begin{subequations}
\label{weakforn*}
\begin{align}
\int_\O h \rho \de\f x\Big|_s^t&=\int_s^t \int_\O h \t \rho + \f m  \cdot \nabla \rho \de \f x \de \tau, \label{mass*}
\\
\begin{split}
\int_\O \f m\cdot \f \varphi \de \f x \Big|_s^t
&= \int_s^t \int_\O \f m   \t \f \varphi + \left (\frac{\f m  \otimes \f m}{h}  \right ) : (\nabla \f \varphi )_{\sym}+ a h^\gamma    (\di \f \varphi)  \de \f x \de \tau \\
&\qquad + \int_s^t \int_{\T} \nabla \f \varphi : \de \mathfrak R (\tau) \de \tau
\end{split}
\label{momentum*}
\end{align}
are fulfilled for all $\rho \in \C^1 (\T\times[0,T])$  and 
$\f\varphi \in \C^1 (\T\times[0,T];\R^d )$,
and for a.a.~$s,t\in(0,T)$, including $s=0$ with $(\dens (0), \mom (0))= ( \dens_0,\mom_0)$. 
The function $E$ is non-increasing and satisfies $E(0+)=\mathcal{E}(h_0,\f m_0)$ and 
\begin{equation}
\mathcal{E}(h (t), \f m(t) ) 
+ c_{\mathfrak R}\int_{\T}\de \tr[\mathfrak R(t)]     \leq E(t)
\label{energy*}
\end{equation}
\end{subequations}
for a.a.~$t\in(0,T)$ and
a constant $ c_{\mathfrak R}\geq 0$.
\end{definition} 
Now we show that energy-variational solutions to \eqref{eq:comprEuler} 
coincide with dissipative weak solutions in the above sense. 

\begin{theorem}\label{thm:measval.CompEul}
Let $\pres(\dens)=a\dens^\gamma$ with $a>0$ and $\gamma>1$,
and let $ ( h_0, \f m_0)\in L^1(\T)\times L^1(\T;\R^d)$ satisfy $ \E(h _0,\f m_0 ) < \infty$. 
Consider a tuple
$(h,\f m, E )\in L^\infty(0,T; L^\gamma ( \T )) \times L^\infty(0,T; L^q (\T ;\R^d ))  \times \BV$ 
with $q=2\gamma/(1+\gamma)$.
Then $(h,\f m, E )$ is an 
energy-variational solution in the sense of Definition~\ref{def:envarEUL*}
with $\alpha=\gamma-1$ if and only if it 
is a dissipative weak solution  in the sense of Definition~\ref{def:weakEul} with $c_{\mathfrak{R}}= \min\left \{ \frac{1}{2}, \frac{1}{d(\gamma-1)} \right \} $.
\end{theorem}

\begin{proof}
Let $(h, \f m, E )$ be an energy-variational solution in the sense of Definition~\ref{def:envarEUL*}. 
Since the regularity weight $\mathcal{K}=\mathcal K_\alpha$ given in~\eqref{K:CompEul} is homogeneous of degree one, we may apply Proposition~\ref{prop:equality} in order to infer $E \big|_{s}^t \leq 0$ and
\begin{multline}
- \left [   \int_{\T}  h
\rho
 +  \f m  \cdot \f \varphi
  \de \f x  \right ] \Big|_{s}^t 
+ \int_s^t 
\int_{\T} h \t \rho
+\f m \cdot  \nabla \rho
\de \f x 
\de s
\\
+ \int_s^t \int_{\T}
\f m \cdot \t \f \varphi 
+  \left (\frac{ \f m  \otimes \f m}{h} + I h ^ \gamma \right ): \nabla \f \varphi 
\de \f x 
+ \mathcal{K}\left (\f \varphi \right )\left [\E(\dens,\mom )- E\right ]
\de  s
\leq 0
\,\label{relenEulmini}
\end{multline}
for every $(\rho , \f \varphi ) \in \C^1([0,T];\C^1(\T;\R^m))$ and a.e. $s<t\in[0,T]$, where $E(0+) = \E(\dens_0,\mom_0)$. 
For the choice 
$\f \varphi \equiv 0$  we infer~\eqref{mass*}, 
but first merely with an inequality sign.
However, since $\rho$ varies in a linear space, 
the equality~\eqref{mass*} follows immediately. 
Choosing $\rho=0$ in~\eqref{relenEulmini} instead implies
\begin{align}
\begin{aligned}
- \int_{\T} \f m \cdot \f \varphi \de \f x \Big|_s^t + \int_s^t \int_{\T}
\f m \cdot \t \f \varphi 
&+  \left (\frac{ \f m  \otimes \f m}{h} + I h ^ \gamma \right ): \nabla \f \varphi 
\de \f x \de \tau \\
&\qquad\qquad
\leq \int_s^t\mathcal{K}( \f \varphi ) \left [ E- \E(h ,\f m) \right ] \de \tau \,.
\end{aligned}
\label{esttimederiEul}
\end{align}
The left-hand side of \eqref{esttimederiEul} defines 
a linear functional $\f l$ by 
\begin{align}
\langle\f l , \f \varphi \rangle ={}&
- \int_{\T} \f m \cdot \f \varphi \de \f x \Big|_0^T + \int_0^T \int_{\T}
\f m \cdot \t \f \varphi 
+  \left (\frac{ \f m  \otimes \f m}{h} + I h ^ \gamma \right ): \nabla \f \varphi 
\de \f x \de \tau 
\end{align}
for $\f \varphi\in \mathcal V$, where
\[
\mathcal V:=\{\f \varphi\in \C^1(\O\times[0,T];\R^d) \mid\int_\O\f\varphi\de x =0\}.
\]
We define the sublinear mapping $\mathfrak p$ by
\[
\begin{aligned}
&\mathfrak p : L^1(0,T;\C(\O ; \R^{d\times d}_{\sym}))  \ra \R, \\
&\mathfrak p(\Phi) := \max \setl{ 2,d(\gamma-1)} \int_0^T \| (\Phi)_{-}\|_{\C(\T;\R^{d\times d})}(E - \mathcal{E}(h,\f m))\de t.
\end{aligned}
\]
From~\eqref{esttimederiEul}, we infer the estimate $\langle \f l, \f \varphi \rangle \leq \mathfrak p(\mathcal{I}(\f \varphi))$ for all $\f \varphi \in \mathcal{V}$. 
Lemma~\ref{lem:hahn}  shows the existence of an element
$\mathfrak R \in L^\infty_{w^*}(0,T;\mathcal{M}(\O ; \R_{\sym}^{d\times d}
)) $
satisfying 
\[
\forall\Phi\in L^1(0,T; \C(\O; \R_{\sym}^{d\times d })):\ \langle -\mathfrak R, \Phi\rangle \leq \mathfrak p(\Phi),
\qquad
\forall\f \varphi \in \mathcal V: \ \langle -\mathfrak R, \nabla \f \varphi\rangle = \langle \f l , \nabla \f \varphi \rangle.
\]
As for the incompressible Euler equations 
(see the proof of Theorem \ref{thm:incompEuler}),
we show that 
$\mathfrak R \in L^\infty_{w^*}(0,T;\mathcal{M}(\O ; \R_{\sym,+}^{d\times d}))$
and that \eqref{momentum*} holds for all $\f \varphi\in \C^1(\O\times[0,T]; \R ^d )$.
Considering
$\Phi(x,t)= -\psi(t)I$ for some $\psi\in\C_0^{1}([0,T))$ with $\psi\geq 0$,
we further have
\[
\int_0^T \psi(t) \int_{\T} \de \tr[\mathfrak R] (t) \de t
=\langle -\mathfrak R, \Phi\rangle 
\leq \mathfrak p(\Phi)
=\max\big \{2, d(\gamma-1)  \big\}\int_0^T \psi(t) (E - \mathcal{E}(h,\f m))\de t.
\]
Since $\psi\geq 0$ is arbitrary,  this directly implies \eqref{energy*} 
for a.a.~$t\in(0,T)$,
where

In order to infer the converse implication, 
let $(h,\f m, E )$ be a dissipative weak solution in the sense of Definition~\ref{def:weakEul}. 
Adding $ E |_{s}^{t}\leq 0$ for $s<t$ and
the identities~\eqref{mass*} and~\eqref{momentum*} with 
$\rho = - \phi$ and $\f \varphi = - \f \psi$, we infer
\begin{equation}
\label{eq:distoev.compr}
\begin{aligned}
&\left [ E - \int_{\T} h \phi 
+ \f m  \cdot \f \psi 
 \de \f x  \right ] \Big|_s^t + \int_s^t 
\int_{\T} h \t \phi
+\f m \cdot  \nabla  \phi 
\de \f x 
\de \tau 
\\& \quad 
+ \int_s^t \int_{\T}
\f m \cdot \t \f \psi 
+  \left (\frac{ \f m  \otimes \f m}{h} + h^\gamma I \right ): \nabla \f \psi 
\de \f x  \de \tau  
+ \int_s^t \int_{\T} \nabla \f \psi : \de \mathfrak R (t) \de \tau 
\leq 0\,
\end{aligned}
\end{equation}
 for a.e.~$s<t\in(0,T)$ and all test functions 
$(\phi,\f \psi )\in \C^1(\T\times [0,T])\times\C^1(\T\times [0,T];\R^m)$. 
From $ \mathfrak R \in L^\infty_{w^*}(0,T;\mathcal M(\T;\R^{d\times d}_{\sym,+}))$,  
the duality between spectral norm and trace norm,
H\"older's inequality, and inequality~\eqref{energy*}, we infer 
\begin{align*}
\int_{\T} \nabla \f \psi : \de \mathfrak R 
&\geq  \int_{\T} (\nabla \f \psi)_{\sym,-}  : \de \mathfrak R 
\geq  -\|(\nabla \f \psi)_{\sym,-} \|_{L^\infty(\T;\R^{d\times d})}  \int_{\T} I : \de \mathfrak R 
\\
& \geq  \|(\nabla \f \psi)_{\sym,-} \|_{L^\infty(\T;\R^{d\times d})} c_{\mathfrak{R}}^{-1} \bb{ \E(\dens , \mom)-E} \,
= \mathcal K_\alpha(\f \psi) \bb{ \E(\dens , \mom)-E}
\end{align*}
a.e.~in $(0,T)$,
where $\alpha=\gamma-1$.
Using these estimates in \eqref{eq:distoev.compr} yields
\eqref{relenEul*}.
\end{proof}
\subsection{Relative entropy inequality and weak-strong uniqueness}
It is readily shown that the Hypothesis~\ref{hypo:smooth} 
is fulfilled for the compressible Euler equations~\eqref{eq:comprEuler},
and that the weak-strong uniqueness principle of Corollary~\ref{cor:uni} holds.
In particular, relation~\eqref{eq:entropyflux.new} was already observed in Remark~\ref{rem:entropyflux}. Nevertheless, the calculation of the relative entropy inequality~\eqref{inuniqu}  for this non-quadratic energy remains a nonstandard task,
and we exemplify it here for the reader's convenience. 
All calculations are done along the lines of Proposition~\ref{prop:weakstrong}.
Note that during the calculations only~\eqref{eq:pres.from.pot} is used, but in order to 
derive weak-strong uniqeness, we explicitly need~\eqref{eq:pot.properties} and  
the assumptions of Lemma~\ref{lem:nonlin.convex.compEuler}.

The relative total entropy $\mathcal{R}$ is given by
\begin{align*}
\mathcal{R}(h,\f m | \th,\tm ) ={}& \int_\O \frac{|\f m|^2}{2h} - \frac{|\tm|^2}{2h}- \frac{\tm}{\th}\cdot\left ( \f m-\tm\right ) +\frac{|\tm|^2}{2\th^2}( h -\th) \de \f x\\& + \int_\O   P(\dens)-P(\th)-P'(\th)(\dens-\th)
\de \f x 
\\
={}& \int_\O \frac{h}{2} \left |\frac{\f m}{h} - \frac{\tm}{\th}\right |^2+P(\dens)-P(\th)-P'(\th)(\dens-\th)
\de \f x 
\,,
\end{align*}
the system operator $\mathcal{A}$ by 
\[
\mathcal{A}(\th, \tm) = \begin{pmatrix}
\t \th + \di  \tm \\
\t \tm  + \di \left (\frac{\tm\otimes \tm}{\th} +  \pres(\th) I \right )
\end{pmatrix},
\]
and the relative Hamiltonian is defined via
\[
\begin{aligned}
&\mathcal{W}(h,\f m|\th,\tm) \\
&= \int_\O \left [h \left (\frac{\f m}{h} -\frac{\tm}{\th}\right ) \otimes \left ( \frac{\f m}{h} -\frac{\tm}{\th}\right )+ 
 \left ( 
 p(\dens)-p(\th)-p'(\th)(\dens-\th)
 \right ) I \right ]: \left (\nabla \left (\frac{\tm}{\th}\right )\right )_{\sym} \!\!\de \f x \\
 &\qquad+ \mathcal{K}\left(\frac{\f m}{h}\right )  \mathcal{R}(h,\f m| \th,\tm) \,,
\end{aligned}
\]
where the regularity measure $\mathcal{K}$ is given as above.
\begin{proposition}
Let $(\dens,\mom)$ be energy-variational solution in the sense of Definition~\ref{def:envarEUL*}
with initial value $(\dens_0,\mom_0)$.
Then $(\dens,\mom)$
fulfills the relative entropy inequality 
\begin{multline}
\left [\mathcal{R}(h,\f m | \th,\tm ) + E - \E(h,\f m ) \right ]\Big|_s^t - \int_s^t \mathcal{K}\left (\frac{\tm}{\th}\right ) \left [\mathcal{R}(h,\f m | \th,\tm ) + E - \E(h,\f m ) \right ] \de \tau 
\\
+\int_s^t \mathcal{W}(h,\f m| \th,\tm) + \left \langle \mathcal{A}(\th,\tm) , 
\begin{pmatrix}
P''(\th)(h-\th) - \frac{h\tm}{\th^2}\left (\frac{\f m }{h}-\frac{\tm}{\th}\right ) \\
\frac{h}{\th}\left ( \frac{\f m}{h}-\frac{\tm}{\th}\right ) 
\end{pmatrix}
\right \rangle \de \tau \leq 0 \,
\label{relenEuler}
\end{multline}
for all $\th \in \C^1(\T \times [0,T]; (0,\infty))$ and $\tm \in \C^1(\T\times [0,T]; \R^d)$.
Moreover, if $p(h)=a h^\gamma$, 
and $(\th,\tm)$ is a (classical) solution to \eqref{eq:comprEuler} with
$(\th,\tm)(0)=(\th_0,\tm_0)$,
then $(h,\f m)=(\th,\tm)$.
\end{proposition}
\begin{proof} 
First we calculate the second derivative of the entropy function $\eta$ and mulitply it with the difference $(h - \th , \f m -\tm)$, which implies
\begin{align*}
D ^2 \eta ( \th,\tm) \begin{pmatrix}
h-\th \\
\f m - \tm 
\end{pmatrix} ={}& \begin{pmatrix}
P''(\th )
+ \frac{|\tm|^2}{\th^3}& - \frac{\tm^T}{\th^2} 
\\-\frac{\tm }{\th^2}& \frac{1}{\th}I
\end{pmatrix}
\begin{pmatrix}
h-\th \\
\f m - \tm 
\end{pmatrix}\\={}& \left ( 
P''(\th)
(h-\th) - \frac{h\tm}{\th^2}\left (\frac{\f m }{h}-\frac{\tm}{\th}\right ) , \frac{h}{\th}\left ( \frac{\f m}{h}-\frac{\tm}{\th}\right ) \right )^T \,.
\end{align*}
This gives the term the system operator~$\mathcal{A}(\th,\tm)$ is tested with in~\eqref{relenEuler}. 
For this, we observe by some calculations and the identity $\frac{p'(\th) }{\th}= P''(\th)$ that
\begin{equation}
\begin{aligned}
\Big \langle \mathcal{A}(\th,\tm) &, 
\begin{pmatrix}
P''(\th)
 (h-\th) - \frac{h\tm}{\th^2}\left (\frac{\f m }{h}-\frac{\tm}{\th}\right ) \\
\frac{h}{\th}\left ( \frac{\f m}{h}-\frac{\tm}{\th}\right ) 
\end{pmatrix}
\Big \rangle 
\\
&= 
\int_\O  \t \th P''(\th) (h-\th) - \t \th \frac{h\tm}{\th^2}\left (\frac{\f m }{h}-\frac{\tm}{\th}\right ) \de \f x 
\\
&\qquad
+ \int_\O 
\di \tm
 \left (
P''(\th) (h-\th) - \frac{h\tm}{\th^2}\left (\frac{\f m }{h}-\frac{\tm}{\th}\right ) \right ) \de \f x 
\\
&\qquad
+ \int_\O (\t \th \frac{\tm}{\th} + \th \t \frac{\tm}{\th})  \frac{h}{\th}\left ( \frac{\f m}{h}-\frac{\tm}{\th}\right )  \de \f x \\
&\qquad
+ \int_\O \left ( 
\di \tm \frac{\tm}{\th}+ \tm \nabla \left ( \frac{\tm}{\th}\right )
+  \nabla \th  p'(\th) \right ) \frac{h}{\th}\left ( \frac{\f m}{h}-\frac{\tm}{\th}\right ) \de \f x 
\\
&= 
\int_\O \left [\t\th + \di \tm \right ] 
P''(\th) (h-\th)  \de \f x 
\\
&\qquad
+ \int_\O \left [ \t \left (\frac{\tm}{\th} \right )+ \frac{\tm}{\th} \cdot \nabla \left ( \frac{\tm}{\th}\right )
+  \nabla \th  P''(\th)  \right ]  h\left ( \frac{\f m}{h}-\frac{\tm}{\th}\right ) \de \f x \,.\label{soloperator}
\end{aligned}
\end{equation}
Adding and substracting the energy $\E(h,\f m)$ in the first term 
of the energy-variational formulation~\eqref{relenEul*}
and choosing $\rho = 
P'(\th)- \frac{|\tm|^2}{2\th^2} $ and $\f \varphi = \frac{\tm}{\th}$ 
as test functions, we further observe that
\[
\begin{aligned}
&\left [ E- \E(h,\f m)  + \int_{\T}
P(\dens)-P'(\th)h  + \frac{|\f m|^2}{2h}  - \f m  \cdot \frac{\tm}{\th} + \frac{| \tm|^2}{2\th^2} h \de \f x  \right ] \Big|_s^t 
\\
&\quad
+ \int_s^t \!\!
\int_{\T} h \t  \th 
P''(\th)- h \t \frac{\tm}{\th}\cdot \frac{\tm}{\th} 
+
 \f m \cdot   \nabla\th 
 P''(\th) - \f m \cdot \nabla \frac{ \tm}{\th} \cdot \frac{ \tm}{\th}
\de \f x 
\de s
\\
&\quad
+ \int_s^t \!\!\int_{\T}
\f m \cdot \t \frac{\tm}{\th} +  \left (\frac{ \f m  \otimes \f m}{h} +  
p(h) I\right ): \nabla \frac{\tm}{\th}
\de \f x 
+ \mathcal{K}\left ( \frac{\tm}{\th}\right ) \left [ \E(h ,\f m) - E\right ] 
\!\de  s
\leq 0
\,.
\end{aligned}
\]
We now invoke the identity 
\[
h \t \th P''(\th) = \t \th P''(\th)\th + \t \th P''(\th) (h-\th)= 
- \t \bb{ P(\th) - P'(\th) \th } + \t \th P''(\th) (h-\th)
\]
to introduce the relative energy in the first line.
Subsequently, we use equation~\eqref{soloperator}
to deduce
\[
\begin{aligned}
&\left [ E- \E(h,\f m)  +\mathcal{R}( h,\f m | \th,\tm)   \right ] \Big|_s^t+ \int_s^t
\mathcal K \left ( \frac{\tm}{\th}\right )  \left [ \E(h ,\f m) - E\right ] 
\de  s 
\\
&\quad
+ \int_s^t\!\! \int_{\T}
 \left (\frac{ \f m  \otimes \f m}{h} + 
 p(\dens ) I \right ): \nabla \frac{\tm}{\th}- \f m \cdot \nabla \frac{ \tm}{\th} \cdot\frac{ \tm}{\th} - \frac{\tm}{\th} \cdot\nabla \frac{\tm}{\th}\cdot  \f m 
\de \f x \de \tau
\\
&\quad
+ \int_s^t\!\!\int_\O 
\left [  \frac{\tm}{\th} \cdot \nabla \left ( \frac{\tm}{\th}\right )
+ 
 \nabla \th  
 P''(\th)\right ]  \cdot h\frac{\tm}{\th}
-\left [ \th \di \frac{\tm}{\th} +\nabla \th \cdot \frac{\tm}{\th} \right ] 
P''(\th) (h-\th)
\de \f x \de \tau
\\
&\quad
+\int_s^t \left \langle \mathcal{A}(\th,\tm) , 
\begin{pmatrix}
P''(\th) (h-\th) - \frac{h\tm}{\th^2}\left (\frac{\f m }{h}-\frac{\tm}{\th}\right ) \\
\frac{h}{\th}\left ( \frac{\f m}{h}-\frac{\tm}{\th}\right ) 
\end{pmatrix}
\right \rangle \de \tau 
\leq 0\,.\label{lastnumber}
\end{aligned}
\]
With the identity $\th  P''(\th)=p'(\th)$ and 
integration by parts,
the second and the third line can be transformed to
\[
\begin{aligned}
\int_s^t\!\!\int_\O \di \frac{\tm}{\th} \left ( 
p(\dens){-}p(\th) {-}p'(\th)(\dens{-}\th)
 \right )
+ h \left ( \frac{\f m}{h}- \frac{\tm}{\th}\right ) \otimes \left ( \frac{\f m}{h}- \frac{\tm}{\th}\right ) : \left (\nabla \frac{\tm}{\th} \right )_{\sym}\!\!\de \f x \de\tau \,,
\end{aligned}
\]
which implies the relative entropy inequality~\eqref{relenEuler}. 
The weak-strong uniqueness principle now follows as in the proof of 
Corollary~\ref{cor:uni},
where we required that 
the relative entropy $\mathcal{R}$ and the relative Hamiltonian $\mathcal{W}$ are non-negative.
For this purpose, we assume $p(h)=a h^\gamma$ again,
which satisfies~\eqref{eq:pot.properties} and  
the assumptions of Lemma~\ref{lem:nonlin.convex.compEuler}.
\end{proof}

\end{document}